\pdfoutput=1
\documentclass[12pt]{amsart}

\usepackage[margin=1in,bottom=1in,top=1in,twosideshift=0in]{geometry}
\usepackage{ifpdf}
\usepackage{amssymb}
\usepackage{amsthm}
\usepackage[usenames]{color}
\usepackage{calc}
\usepackage{graphicx}
\usepackage{bm} 

\usepackage{comment}
\excludecomment{long-version}
\includecomment{short-version}

\definecolor{webgreen}{rgb}{0,.5,0}
\definecolor{webbrown}{rgb}{.6,0,0}

\newcommand{\esup}{\mathop{\text{ess\;sup}}}    
\renewcommand{\P}{\mathbb{P}}                   
\newcommand{\E}{\mathbb{E}}                     
\newcommand{\Em}{\mathcal{E}}                   
\newcommand{\F}{\mathcal{F}}                    
\newcommand{\Fb}{\mathbb{F}}                    
\newcommand{\Mh}{\mathcal{M}}                      
\newcommand{\R}{\mathbb{R}}                     
\newcommand{\I}{\bm{1}}                         
\newcommand{\II}[1]{\bm{1}_{\left\{#1\right\}}} 
\newcommand{\bpi}{\pi}                    
\newcommand{\bPi}{\Pi}                          
\newcommand{\e}{e}
\newcommand{\T}{\mathbb{T}}
\newcommand{\M}{\mathbb{M}}                     

\newtheorem{theorem}{theorem}
\newtheorem{proposition}[theorem]{Proposition}
\newtheorem{lemma}[theorem]{lemma}
\newtheorem{definition}[theorem]{Definition}
\newtheorem{corollary}[theorem]{Corollary}
\newtheorem{remark}[theorem]{Remark}

\numberwithin{equation}{section}
\numberwithin{theorem}{section}

\date{August 2006}
\author{SAVAS DAYANIK}
\address[S.\ Dayanik and C.\ Goulding]{Department of Operations Research and
  Financial Engineering, and the Bendheim Center for Finance\\
  Princeton University, Princeton, NJ 08544}
\email{sdayanik@princeton.edu,cgouldin@princeton.edu}

 \author{CHRISTIAN GOULDING}

\author{H.\ Vincent POOR}
\address[H.\ V.\ Poor]{School of Engineering and Applied Science,
  Princeton University, Princeton, NJ 08544}
\email{poor@princeton.edu}

\begin{document}
\begin{abstract}
  Sequential change diagnosis is the joint problem of detection and
  identification of a sudden and unobservable change in the
  distribution of a random sequence.
  In this problem, the
  common probability law of a sequence of i.i.d.\ random variables
  suddenly changes at some disorder time to one of finitely many
  alternatives.  This disorder time marks the start of a new regime,
  whose fingerprint is the new law of observations.  Both the disorder
  time and the identity of the new regime are unknown and
  unobservable.  The objective is to detect the regime-change as soon
  as possible, and, at the same time, to determine its identity as
  accurately as possible.  Prompt and correct diagnosis is crucial for
  quick execution of the most appropriate measures in response to the
  new regime, as in fault detection and isolation in industrial
  processes, and target detection and identification in national defense.
  The problem is formulated in a Bayesian framework.
  An optimal sequential decision strategy is found, and an accurate
  numerical scheme is described for its implementation.  Geometrical
  properties of the optimal strategy are illustrated via numerical
  examples.  The traditional problems of Bayesian change-detection and
  Bayesian sequential multi-hypothesis testing are
  solved
  as special cases.
  In addition, a solution is obtained for the problem
  of detection and identification of component failure(s)
  in a system with suspended animation.

\end{abstract}

\title{Bayesian sequential change diagnosis}
\maketitle

\section{Introduction}
\label{sec:Introduction}

Sequential change diagnosis is the joint problem of detection and
identification of a sudden change in the
distribution of a random sequence.  In this problem, one observes
a sequence of i.i.d.\ random variables $X_1, X_2, \ldots$, taking
values in some measurable space $(E,\Em)$. The common probability
distribution of the $X$'s is initially some known probability
measure $\P_0$ on $(E,\Em)$, and, in the terminology of
statistical process control, the system is said to be ``in
control.''  Then, at some \emph{unknown} and \emph{unobservable}
disorder time $\theta$, the common probability distribution
changes suddenly to another probability measure $\P_{\mu}$ for
some \emph{unknown} and \emph{unobservable} index $\mu\in \Mh
\triangleq \{1,\ldots,M\}$, and the system goes ``out of control."
The objective is to detect the change as quickly as possible, and,
at the same time, to
identify
the new probability distribution as
accurately as possible, so that the most suitable actions can be
taken with the least delay.

Decision strategies for this problem have a wide array of
applications, such as fault detection and isolation in industrial
processes, target detection and identification in national
defense, pattern recognition and machine learning, radar and sonar
signal processing, seismology, speech and image processing,
biomedical signal processing,
finance, and insurance.
 For example, suppose we perform a quality
test on each item produced from a manufacturing process
consisting of several complex processing components (labeled $1,
2,\ldots, M$).  As long as each processing component is operating
properly, we can expect the distribution of our quality test
statistic to be stationary. Now, if there occurs a sudden fault in
one of the processing components, this can change the distribution
of our quality test statistic depending on the processing
component which caused the fault.  It may be costly to continue
manufacture of the items at a substandard quality level, so we
must decide when to (temporarily) shut down the manufacturing
process and repair the fault. However, it may also be expensive to
dissect each and every processing component in order to identify
the source of the failure and to fix it. So, not only do we want
to detect quickly when a fault happens, but, at the same time we
want also to identify accurately which processing component is the
cause.  The time and the cause of the fault will be distributed
independently according to a geometric and a finite distribution,
respectively, if each component fails independently according to
some geometric distributions, which is a reasonable assumption for
highly reliable components;
see Section~\ref{sec:suspended-animation}.
 As another example, an insurance
company may monitor reported claims not only to detect a change in
its risk exposure, but also to assess the nature of the change so
that it can adjust its premium schedule or re-balance
appropriately its portfolio of reserves to hedge against a
different distribution of loss scenarios.

Sequential change diagnosis can be viewed as the fusion of two
fundamental areas of sequential analysis: change detection and
multi-hypothesis testing.  In traditional change detection problems,
$M=1$ and there is only one change distribution, $\P_1$; therefore,
the focus is exclusively on detecting the change time, whereas in
traditional sequential multi-hypothesis testing problems, there is no
change time to consider.  Instead, every observation has common
distribution $\P_\mu$ for some unknown $\mu$, and the focus is
exclusively on the inference of $\mu$.  Both change detection and
sequential multi-hypothesis testing have been studied extensively. For
recent reviews of these areas, we refer the reader to Basseville and
Nikiforov \cite{MR1210954}, Dragalin, Tartakovsky and Veeravalli
\cite{MR1725130,MR1768555}, and Lai \cite{MR1844531}, and the
references therein.

However, the sequential change diagnosis problem involves key
trade-off decisions not taken into account by separately applying
techniques for change detection and sequential multi-hypothesis
testing.  While raising an alarm as soon as the change occurs is
advantageous for the change detection task, it is undesirable for
the isolation task because the longer one waits to raise the
alarm, the more observations one has to use for inferring the
change distribution.  Moreover, the unknown change time
complicates the isolation task, and, as a result, adaptation of
existing sequential multi-hypothesis testing algorithms is
problematic.

The theory of sequential change diagnosis has not been broadly
developed.  Nikiforov~\cite{Nikiforov1995} provides the first
results for this problem, showing asymptotic optimality for a
certain non-Bayesian approach, and Lai~\cite{Lai2000} generalizes
these results through the development of information-theoretic
bounds and the application of likelihood methods.  In this paper,
we follow a Bayesian approach to reveal a new sequential decision
strategy for this problem, which incorporates a~priori knowledge
regarding the distributions of the change time $\theta$ and of the
change index $\mu$. We prove that this strategy is optimal and we
describe an accurate numerical scheme for its implementation.

In Section \ref{sec:Problem-statement} we formulate precisely the
problem in a Bayesian framework, and in Section
\ref{sec:Reformulation} we show that it can be reduced to an optimal
stopping of a Markov process whose state space is the standard
probability simplex.  In addition, we establish a simple recursive
formula that captures the dynamics of the process and yields a
sufficient statistic fit for online tracking.

In Section \ref{sec:dynamic-programming-solution} we use optimal
stopping theory to substantiate the optimality equation for the value
function of the optimal stopping problem.  Moreover, we prove that
this value function is bounded, concave, and continuous on the
standard probability simplex.  Furthermore, we prove that the optimal
decision strategy uses a finite number of observations on average and
we establish some important characteristics of the associated optimal
stopping/decision region.  In particular, we show that the optimal
stopping region of the state space for the problem consists of $M$
non-empty, convex, closed, and bounded subsets.  Also, we consider a
truncated version of the problem that allows at most $N$ observations
from the sequence of random measurements.  We establish an explicit
bound (inversely proportional to $N$) for the approximation error
associated with this truncated problem.

In Section \ref{sec:Special-cases} we show that the separate problems
of change detection and sequential multi-hypothesis testing are solved
as special cases of the overall joint solution.  We illustrate some
geometrical properties of the optimal method and demonstrate its
implementation by numerical examples for the special cases $M=2$ and
$M=3$.  Specifically, we show instances in which the $M$ convex
subsets comprising the optimal stopping region are connected and
instances in which they are not.  Likewise, we show that the
continuation region (i.e., the complement of the stopping region) need
not be connected.  We provide a solution to the problem of detection
and identification of component failure(s) in a system with suspended
animation. Finally, we outline in Section
\ref{sec:computer-implementation} how the change-diagnosis algorithm
may be implemented with a computer in general. Proofs of most results
are deferred to the Appendix.

\section{Problem statement}
\label{sec:Problem-statement}

Let $(\Omega,\F, \P)$ be a probability space hosting random variables
$\theta:\Omega\mapsto\{0,1,\ldots\}$ and $\mu:\Omega\mapsto \Mh
\triangleq \{1,\ldots,M\}$ and a process $X = (X_n)_{n\geq1}$ taking
values in some measurable space $(E,\Em)$. Suppose that for every
$t\ge 1$, $i \in \Mh$, $n\ge 1$, and $(E_k)^n_{k=1}\subseteq \Em $
\begin{multline}
  \label{eq:finite-dimensional-distributions}
  \P\left\{\theta=t, \mu=i, X_1 \in E_1,\ldots, X_n \in E_n \right\}  \\
  = (1-p_0) (1-p)^{t-1} p \nu_i \prod_{1\le k \le (t-1) \land n}
  \P_0(E_k) \prod_{t\lor 1\le \ell \le n} \P_i(E_{\ell})
\end{multline}
for some given probability measures $\P_0,\P_1,\ldots,\P_M$ on
$(E,\Em)$, known constants $p_0\in[0,1]$, $p\in(0,1)$, and
$\nu_i>0,i\in \Mh$ such that $\nu_1+\cdots+\nu_M = 1$, where $x\wedge
y\triangleq \min\{x,y\}$ and $x\vee y\triangleq \max\{x,y\}$.  Namely,
$\theta$ is independent of $\mu$; it has a zero-modified geometric
distribution with parameters $p_0$ and $p$ in the terminology of
Klugman, Panjer, and Willmot \cite[Sec.\ 3.6]{MR1490300}, which
reduces to the standard geometric distribution with success
probability $p$ when $p_0=0$. Moreover, $\nu_i$ is the probability
that the change type $\mu$ is $i$ for every $i=1,\ldots,M$.

Conditionally on $\theta$ and $\mu$, the random variables $X_n$, $n\ge
1$ are independent; $X_1,\ldots, X_{\theta-1}$ and $X_{\theta},
X_{\theta+1},\ldots$ are identically distributed with common
distributions $\P_0$ and $\P_{\mu}$, respectively.  The probability
measures $\P_0,\P_1,\ldots,\P_M$ always admit densities with respect
to some $\sigma$-finite measure $m$ on $(E,\Em)$; for example, we can
take $m = \P_0+\P_1\cdots+\P_M$.  So, we fix $m$ and denote the
corresponding densities by $f_0, f_1,\ldots,f_M$, respectively.

Suppose now that we observe sequentially the random variables $X_n$,
$n\ge 1$. Their common probability density function $f_0$ changes at
stage $\theta$ to some other probability density function $f_{\mu}$,
$\mu\in \Mh$. Our objective is to detect the change time $\theta$ as
quickly as possible \emph{and} isolate the change index $\mu$ as
accurately as possible.  More precisely, given costs associated with
detection delay, false alarm, and false isolation of the change index,
we seek a strategy that minimizes the expected total change detection
\emph{and} isolation cost.

In view of the fact that the observations arrive sequentially, we are
interested in sequential diagnosis schemes. Specifically, let
$\Fb = (\F_n)_{n\geq0}$ denote the natural filtration of the
observation process $X$, where
\begin{align*}
  \F_0=\{\varnothing,\Omega\}\quad\text{and}\quad
  \F_n=\sigma(X_1,\ldots,X_n),\quad n\geq1.
\end{align*}
A \emph{sequential decision strategy} $\delta=(\tau, d)$ is a pair
consisting of a \emph{stopping time (or stopping rule)} $\tau$ of the
filtration $\Fb$ and a \emph{terminal decision rule} $d: \Omega
\mapsto \Mh$ measurable with respect to the history
$\F_{\tau}=\sigma(X_{n\wedge\tau}; n\geq1)$ of observation process $X$
through stage $\tau$.  Applying a sequential decision strategy
$\delta=(\tau,d)$ consists of announcing at the end of stage $\tau$
that the common probability density function has changed from $f_0$ to
$f_d$ at or before stage $\tau$. Let
\begin{align*}
  \Delta \triangleq \{(\tau,d) \mid \tau\in\Fb, \text{ and
    $d\in\F_\tau$ is an $\Mh$-valued random variable}\}
\end{align*}
denote the collection of all such sequential decision strategies
(``$\tau \in \Fb$'' means that $\tau$ is a stopping time of filtration
$\Fb$).  Let us specify the possible losses associated with a
sequential decision strategy $\delta=(\tau,d)\in\Delta$ as follows:
\begin{enumerate}
\item \emph{Detection delay loss.}  Let us denote by a fixed positive
  constant $c$ the detection delay cost per period.  Then the expected
  decision delay cost for $\delta$ is $\E[c(\tau-\theta)^+]$, possibly
  infinite, where $(x)^+\triangleq\max\{x,0\}$.
\item \emph{Terminal decision loss.}  Here we identify two cases of
  isolation loss depending on whether or not the change has actually
  occurred at or before the stage in which we announce the isolation
  decision:
  \begin{enumerate}
  \item \emph{Loss due to false alarm.}  Let us denote by $a_{0j}$ the
    isolation cost on $\{\tau<\theta, d=j\}$ for every $j\in \Mh$.
    Then the expected false alarm cost for $\delta$ is
    $\E[a_{0d}\II{\tau<\theta}]$.
  \item \emph{Loss due to false isolation.}  Let us denote by $a_{ij}$
    the isolation cost on the event $\{\theta\leq\tau<\infty, d=j,
    \mu=i\}$ for every $i,j\in\Mh$.  Then the expected false isolation
    cost for $\delta$ is $\E[a_{\mu d}\II{\theta\leq\tau<\infty}]$.
  \end{enumerate}
  Here, $a_{ij}, i,j\in \Mh$ are known nonnegative constants, and
  $a_{ii}=0$ for every $i\in \Mh$; i.e., no cost incurred for making a
  correct terminal decision.
\end{enumerate}

Accordingly, for every sequential decision strategy $\delta=(\tau,
d)\in\Delta$, we define a \emph{Bayes risk function}
\begin{align}
  R(\delta) = c\,\E[(\tau-\theta)^+] + \E[a_{0
    d}\II{\tau<\theta}+a_{\mu
    d}\II{\theta\leq\tau<\infty}]\label{E:BayesRiskUnderP}
\end{align}
\noindent as the expected diagnosis cost: the sum of the expected
detection delay cost and the expected terminal decision cost upon
alarm.  The problem is to find a sequential decision strategy
$\delta=(\tau,d)\in\Delta$ (if it exists) with the \emph{minimum Bayes
  risk}
\begin{align}
  R^* \triangleq \inf_{\delta\in\Delta} R(\delta).\label{E:UDef1}
\end{align}

\section{Posterior analysis and formulation as an optimal stopping
  problem}
\label{sec:Reformulation}

In this section we show that the Bayes risk function in
(\ref{E:BayesRiskUnderP}) can be written as the expected value of the
running and terminal costs driven by a certain Markov process.  We use
this fact to recast the minimum Bayes risk in (\ref{E:UDef1}) as a
Markov optimal stopping problem.

Let us introduce the posterior probability processes
\begin{align*}
  \Pi_n^{(0)} &\triangleq
  \P\{\theta>n\,|\,\F_n\}\quad\text{and}\quad
  \Pi_n^{(i)} \triangleq \P\{\theta\leq n,\mu = i\,|\,\F_n\},\quad
  i\in \Mh, \; n\geq 0.
\end{align*}
Having observed the first $n$ observations, $\Pi_n^{(0)}$ is the
posterior probability that the change \emph{has not} yet occurred at
or before stage $n$, while $\Pi_n^{(i)}$ is the posterior joint
probability that the change \emph{has} occurred by stage $n$ and that
the hypothesis $\mu=i$ is correct.  The connection of these posterior
probabilities to the loss structure for our problem is established in
the next proposition. 

\begin{proposition}\label{P:BayesRiskInTermsOfPi}
  For every sequential decision strategy $\delta\in\Delta$, the Bayes
  risk function (\ref{E:BayesRiskUnderP}) can be expressed in terms of
  the process $\bPi\triangleq\{ \bPi_n= (\Pi_n^{(0)}, \ldots,
  \Pi_n^{(M)})\}_{n\geq 0}$ as
  \begin{align*}
    R(\delta) = \E\left[ \sum_{n=0}^{\tau-1}c\,(1-\Pi_n^{(0)})
      +\II{\tau<\infty}\sum_{j=1}^{M}\II{d=j}\sum_{i=0}^{M}
      a_{ij}\Pi_{\tau}^{(i)}\right].
  \end{align*}
\end{proposition}

While our original formulation of the Bayes risk function
(\ref{E:BayesRiskUnderP}) was in terms of the values of the
unobservable random variables $\theta$ and $\mu$, Proposition
\ref{P:BayesRiskInTermsOfPi} gives us an equivalent version of the
Bayes risk function in terms of the posterior distributions for
$\theta$ and $\mu$.  This is particularly effective in light of
Proposition \ref{P:PiProperties}, which we state with the aid of some
additional notation that is referred to throughout the paper.  Let
\begin{align*}
  S^M \triangleq
  \left\{\bpi=(\pi_0,\pi_1,\ldots,\pi_M)\in[0,1]^{M+1}\,\bigm|\,
    \pi_0+\pi_1+\cdots+\pi_M = 1 \right\}
\end{align*}
denote the standard $M$-dimensional probability simplex. Define the
mappings $D_i:S^M \times E \mapsto [0,1], i\in \Mh$ and $D:S^M \times
E \mapsto [0,1]$ by
\begin{align}
  \label{eq:D-mappings}
  D_{i}(\bpi,x) \triangleq \left\{
    \begin{aligned}
      &(1-p)\pi_0 f_0(x), && i=0\\
      &(\pi_i+\pi_0\,p\nu_i) f_i(x), && i\in \Mh
    \end{aligned}
  \right\}, \qquad D(\bpi,x)\triangleq\sum_{i=0}^{M}D_{i}(\bpi,x),
\end{align}
and the operator $\T$ on the collection of bounded functions $f:S^M
\mapsto\R$ by
\begin{align}
  \label{E:T-operator}
  (\T f)(\bpi) &\triangleq \int_{E}
  m(dx)\,D(\bpi,x)\,f\left(\frac{D_0(\bpi,x)}{D(\bpi,x)},\ldots,
    \frac{D_M(\bpi,x)}{D(\bpi,x)}\right) \text{ for every }\bpi\in
  S^M.
\end{align}

\begin{proposition}\label{P:PiProperties}
  The process $\bPi$ possesses the following properties:
  \begin{itemize}
  \item[(a)] The process
    $\bPi^{(0)}\triangleq\{\Pi_n^{(0)},\F_n\}_{n\geq 0}$ is a
    supermartingale, and $\E\,\Pi_n^{(0)} \leq (1-p)^n$ for every
    $n\geq 0$.
  \item[(b)] The process
    $\bPi^{(i)}\triangleq\{\Pi_n^{(i)},\F_n\}_{n\geq 0}$ is a
    submartingale for every $i\in \Mh$.
  \item[(c)] The process
    $\bPi=\{(\Pi_n^{(0)},\ldots,\Pi_n^{(M)})\}_{n\geq 0}$ is a Markov
    process, and
    \begin{align}
      \Pi_{n+1}^{(i)} =
      \frac{D_i(\bPi_n,X_{n+1})}{D(\bPi_n,X_{n+1})},\quad i\in
      \{0\}\cup\Mh ,\quad n\geq 0,\label{E:Pi-Dynamics}
    \end{align}
    with initial state $\Pi_{0}^{(0)} = 1-p_0$ and
    $\Pi_{0}^{(i)}=p_0\nu_i$, $i\in \Mh.$ 
    Moreover, for every bounded function $f:S^M\mapsto\R$ and $n\geq
    0$, we have $\E[f(\bPi_{n+1})|\bPi_n] = (\T f)(\bPi_n)$.
  \end{itemize}
\end{proposition}

\begin{remark}\label{R:PiProperties}
  Since $1=\sum_{i=0}^{M}\Pi_n^{(i)}$, the vector
  $(\Pi_n^{(0)},\ldots,\Pi_n^{(M)})\in S^M$ for every $n\geq 0$.
  Since $\bPi$ is uniformly bounded, the limit
  $\lim_{n\rightarrow\infty}\bPi_n$ exists by the martingale
  convergence theorem.  Moreover,
  $\lim_{n\rightarrow\infty}\Pi_n^{(0)}=0$ a.s.\ by Proposition
  \ref{P:PiProperties}(a) since $p\in(0,1)$.
\end{remark}

Now, let the functions $h, h_1,\ldots,h_M$ from $S^M$ into $\R_+$ be
defined by
\begin{align*}
  h(\bpi)\triangleq \min_{j\in \Mh} h_j(\bpi) \quad \text{and} \quad
  h_j(\bpi) \triangleq \sum_{i=0}^{M} \pi_i\, a_{ij},\quad j\in \Mh,
\end{align*}
respectively.  Then, we note that for every $\delta=(\tau,d)\in
\Delta$, we have
\begin{align*}
  R(\tau, d) &= \E\left[ \sum_{n=0}^{\tau-1}c(1-\Pi_n^{(0)})
    +\II{\tau<\infty}\sum_{j=1}^{M}\II{d=j}h_j(\Pi_{\tau})\right]\\
  &\geq \E\left[ \sum_{n=0}^{\tau-1}c(1-\Pi_n^{(0)})
    +\II{\tau<\infty}h(\Pi_{\tau})\right] = R(\tau,\tilde{d})
\end{align*}
where we define on the event $\{\tau<\infty\}$ the terminal decision
rule $\tilde{d}$ to be any index satisfying
$h_{\tilde{d}}(\Pi_{\tau})=h(\Pi_{\tau})$.  In other words, an optimal
terminal decision depends only upon the value of the $\bPi$ process at
the stage in which we stop.  Note also that the functions $h$ and
$h_1,\ldots,h_M$ are bounded on $S^M$.  Therefore, we have the
following:

\begin{lemma}\label{L:OSP1}
  The minimum Bayes risk (\ref{E:UDef1}) reduces to the following
  optimal stopping of the Markov process $\bPi$:
  \begin{align*}
    R^* = \inf_{(\tau,d)\in\Delta}R(\tau,d) =
    \inf_{(\tau,\tilde{d})\in\Delta}R(\tau,\tilde{d}) =
    \inf_{\tau\in\Fb}
    \,\E\left[\sum_{n=0}^{\tau-1}c\,(1-\Pi_n^{(0)})+\II{\tau<\infty}h(\Pi_\tau)\right].
  \end{align*}
\end{lemma}

We simplify this formulation further by showing that it is enough to
take the infimum over
\begin{align}
  C \triangleq \{\tau\in\Fb\,|\,\tau<\infty \text{ a.s. and }
  \E Y_\tau^-<\infty\},\label{E:C}
\end{align}
where 
\begin{align}
  -Y_n \triangleq \sum_{k=0}^{n-1}c\,(1-\Pi_k^{(0)})+h(\Pi_n),\quad
  n\geq 0\label{E:Yn}
\end{align}
is the minimum cost obtained by making the best terminal decision when
alarm is set at time $n$. Since $h(\cdot)$ is bounded on $S^M$, the
process $\{Y_n, \F_n; n\ge 0\}$ consists of integrable random
variables. So the expectation $\E Y_\tau$ exists for every
$\tau\in\Fb$, and our problem becomes
\begin{align}
  -R^*=\sup_{\tau\in\Fb}\E
  Y_\tau.\label{E:OptimizationProblemTau}
\end{align}

Observe that $\E \tau <\infty$ for every $\tau\in C$ because $\infty >
(1/c)\E Y_\tau^- \geq \E (\tau-\theta)^+ \geq \E (\tau-\theta) \geq \E
\tau -\E \theta \ge \E\tau - (1/p)$.  In fact, we have $\E
Y_\tau>-\infty \Leftrightarrow \E Y_\tau^-<\infty \Leftrightarrow \E
\tau<\infty$ for every $\tau\in\Fb$.  Since
$\sup_{\tau\in\Fb}\E Y_{\tau} \geq \E Y_0 > -h(\Pi_0) >
-\infty$, it is enough to consider $\tau\in\Fb$ such that
$\E\tau <\infty$. Namely, (\ref{E:OptimizationProblemTau}) reduces to
\begin{align}
  -R^*=\sup_{\tau\in C} \E Y_\tau.\label{E:OptimizationProblemC}
\end{align}

\section{Solution via optimal stopping theory}
\label{sec:dynamic-programming-solution}

In this section we derive an optimal solution for the sequential
change diagnosis problem in~\eqref{E:UDef1} by building on the
formulation of~\eqref{E:OptimizationProblemC} via the tools of
optimal stopping theory.

\subsection{The optimality equation}\label{sec:Derive-Opt-Eqn}

We begin by applying the method of truncation with a view of passing
to the limit to arrive at the final result.  For every $N\ge 0$
and $n=0,\ldots,N$, define the sub-collections 
\begin{align*}
  C_n &\triangleq \{\tau \vee n\,|\,\tau\in C\}\quad\text{and}\quad
  C_n^N \triangleq \{\tau \wedge N\,|\,\tau\in C_n\}
\end{align*}
of stopping times in $C$ of \eqref{E:C}.  Note that $C=C_0$. Now,
consider the families of (truncated) optimal stopping
problems 
corresponding to $(C_n)_{n\geq 0}$ and $(C_n^N)_{0\leq n\leq N}$,
respectively, defined by
\begin{align}
  \label{E:Vn-and-VnN}
  -V_n \triangleq \sup_{\tau\in C_n}\E Y_\tau,\; n\ge 0
  \quad\text{and}\quad -V_n^N \triangleq \sup_{\tau\in C_n^N}\E
  Y_\tau,\; 0\le n \le N,\; N\ge 0.
\end{align}
Note that $R^*=V_0$.

To investigate these optimal stopping problems, we introduce
versions of the \emph{Snell envelope} of $(Y_n)_{n\geq 0}$ (i.e.,
the smallest regular supermartingale dominating $(Y_n)_{n\geq
  0}$) corresponding to $(C_n)_{n\geq 0}$ and $(C_n^N)_{0\leq n\leq
  N}$, respectively, defined by
\begin{align}
\label{eq:snell-envelopes}
  \gamma_n &\triangleq \esup_{\tau\in C_n} \E [Y_\tau\,|\,\F_n],\;
  n\ge 0
  \quad\text{and}\quad \gamma_n^N \triangleq \esup_{\substack{\tau\in
      C_n^N}} \E [Y_\tau\,|\,\F_n],\; 0 \le n\le N, \; N\ge 0.
\end{align}
Then through the following series of lemmas, whose proofs are deferred
to the Appendix, we point out several useful properties of these Snell
envelopes. Finally, we extend these results to an arbitrary initial
state vector and establish the optimality equation.  Note that each of
the ensuing (in)equalities between random variables are in the
$\P$-almost sure sense.

First, these Snell envelopes provide the following alternative
expressions for the optimal stopping problems introduced in
\eqref{E:Vn-and-VnN} above.

\begin{lemma}\label{L:Vn-equal-expected-gamma}
  For every $N\ge 0$ and $0\le n\le N$, we have $-V_n = \E \gamma_n$
  and $-V_n^N = \E \gamma_n^N$.
\end{lemma}

Second, we have the following backward-induction equations.

\begin{lemma}\label{L:backward-induction-eqns}
  We have $\gamma_n = \max\{Y_n, \E [\gamma_{n+1}\,|\,\F_n]\}$ for
  every $n\ge 0$. For every $N\ge 1$ and $0\le n \le N-1$, we have
  $\gamma_N^N = Y_N$ and $\gamma_n^N = \max\{Y_n, \E
  [\gamma_{n+1}^N\,|\,\F_n]\}$.
\end{lemma}

We also have that these versions of the Snell envelopes coincide in
the limit as $N\rightarrow\infty$.  That is,

\begin{lemma}\label{L:gamma-equals-gamma-prime}
  For every $n\geq 0$, we have $\gamma_n = \lim_{N\rightarrow\infty}
  \gamma_n^N$.
\end{lemma}

Next, recall from \eqref{E:T-operator} and Proposition
\ref{P:PiProperties}(c) the operator $\T$ and let us introduce the
operator $\M$ on the collection of bounded functions $f:S^M
\mapsto \R_+$ defined by
\begin{align}
\label{eq:operator-M}
  (\M f)(\bpi) \triangleq
  \min\{h(\bpi),c(1-\pi_0)+(\T f)(\bpi)\},\quad\bpi\in S^M.
\end{align}
Observe that $0\leq \M f \leq h$.  That is, $\bpi\mapsto(\M f)(\bpi)$
is a nonnegative bounded function. Therefore, $\M^2 f\equiv \M(\M f)$
is well-defined.  If $f$ is nonnegative and bounded, then $\M^n
f\equiv \M(\M^{n-1} f)$ is defined for every $n\ge 1$, with
$\M^0 f\equiv f$ by definition.  Using operator $\M$, we can express
$(\gamma_n^N)_{0\leq n\leq N}$ in terms of the process $\bPi$ as
stated in the following lemma.

\begin{lemma}\label{L:pg-36}
  For every $N\ge 0$, and $0\le n \le N$, we have
  \begin{align}
    \label{eq:markov-property-of-snell-envelope}
    \gamma_n^N = -c\sum_{k=0}^{n-1}(1-\Pi_k^{(0)})-(\M^{N-n}h)(\Pi_n).
  \end{align}
\end{lemma}

The next lemma shows how the optimal stopping problems can be
rewritten in terms of the operator $\M$. It also conveys the
connection between the truncated optimal stopping problems and the
initial state $\bPi_0$ of the $\bPi$ process.

\begin{lemma}\label{L:pg-38}
  We have
  \begin{itemize}
  \item[(a)] $V_0^N=(\M^N h)(\bPi_0)$ for every $N\geq 0$, and
  \item[(b)] $V_0={\displaystyle\lim_{N\rightarrow\infty}(\M^N
      h)(\bPi_0)}$.
  \end{itemize}
\end{lemma}

Observe that since $\bPi_0\in\F_0=\{\varnothing,\Omega\}$, we have
$\P\{\bPi_0=\bpi\}=1$ for some $\bpi\in S^M$.  On the other hand, for
every $\bpi\in S^M$ we can construct a probability space
$(\Omega,\F,\P_{\bpi})$ hosting a Markov process $\bPi$ with the same
dynamics as in \eqref{E:Pi-Dynamics} and $\P_{\bpi}\{\bPi_0=\bpi\}=1$.
Moreover, on such a probability space, the preceding results remain
valid.  So, let us denote by $\E_{\bpi}$ the expectation with respect
to $\P_{\bpi}$ and rewrite \eqref{E:Vn-and-VnN} as
\begin{align*}
  -V_n(\bpi) \triangleq \sup_{\tau\in C_n}\E_{\bpi} Y_\tau,\; n\ge 0,
  \quad \text{and} \quad -V_n^N(\bpi) \triangleq \sup_{\tau\in
    C_n^N}\E_{\bpi} Y_\tau,\; 0\le n\le N,\; N\ge 0
\end{align*}
for every $\bpi\in S^M$.  Then Lemma \ref{L:pg-38} implies that
\begin{align}
\label{eq:value-functions}
  V_0^N(\bpi)=(\M^N h)(\bpi)\text{ for every }N\geq 0, \quad\text{ and
  }\quad V_0(\bpi)=\lim_{N\rightarrow\infty}(\M^N h)(\bpi)
\end{align}
for every $\bpi\in S^M$.  Taking limits as $N\rightarrow\infty$ of
both sides in $(\M^{N+1}h)(\bpi) = \M(\M^N h)(\bpi)$ and applying
the monotone convergence theorem on the right-hand side yields
$V_0(\bpi) = (\M V_0)(\bpi)$.  Hence, we have shown the following
result.

\begin{proposition}[Optimality equation]\label{P:Dyn-prog-eqn}
  For every $\bpi\in S^M$, we have
  \begin{align}
    V_0(\bpi) = (\M V_0)(\bpi) \equiv
    \min\{h(\bpi),c(1-\pi_0)+(\T V_0)(\bpi)\}.\label{E:Dyn-prog-eqn}
  \end{align}
\end{proposition}


\begin{remark}
  By solving $V_0(\bpi)$ for any initial state $\bpi\in S^M$, we
  capture the solution to the original problem since property (c) of
  Proposition \ref{P:PiProperties} and \eqref{E:OptimizationProblemC}
  imply that
  \begin{align*}
    R^* = V_0(1-p_0,p_0\nu_1,\ldots,p_0\nu_M).
  \end{align*}
\end{remark}

\subsection{Some properties of the value function}\label{sec:V-properties}

Now, we reveal some important properties of the value function
$V_0(\cdot)$ of (\ref{eq:value-functions}).  These results help us
to establish an optimal solution for $V_0(\cdot)$, and hence an
optimal solution for $R^*$, in the next subsection.

\begin{lemma}\label{L:V-concave}
    If $g:S^M \mapsto \R$ is a bounded concave function, then so is $\T g$.
\end{lemma}

\begin{proposition}\label{P:V-concave}
  The mappings $\bpi \mapsto V_0^N(\bpi), N\geq 0$ and $\bpi \mapsto
  V_0(\bpi)$ are concave.
\end{proposition}

\begin{proposition}\label{P:V-convergence-rate}
    For every $N\ge 1$ and $\bpi\in S^M$, we have
        \begin{align*}
            V_0(\bpi)\leq V_0^N(\bpi) \leq
            V_0(\bpi)+\left(\frac{\|h\|^2}{c}+\frac{\|h\|}{p}\right)\frac{1}{N}.
        \end{align*}
        Since $\|h\|\triangleq \sup_{\bpi\in S^M} |h(\bpi)|<\infty$,
        $\lim_{N\rightarrow\infty} \downarrow V_0^N(\bpi) = V_0(\bpi)$
        uniformly in $\bpi\in S^M$.
\end{proposition}

\begin{proposition}\label{P:V0N-continuous}
  For every $N\ge 0$, the function $V_0^N:S^M\mapsto\R_+$ is
  continuous.
\end{proposition}

\begin{corollary}\label{C:V-continuous}
  The function $V_0:S^M \mapsto \R_+$ is continuous.
\end{corollary}

Note that $S^M$ is a compact subset of $\R^{M+1}$, so while continuity
of $V_0(\cdot)$ on the interior of $S^M$ follows from the concavity of
$V_0(\cdot)$ by Proposition \ref{L:V-concave}, Corollary
\ref{C:V-continuous} establishes continuity on all of $S^M$, including
its boundary.

\subsection{An optimal sequential decision
  strategy}\label{sec:optimal-soln}

Finally, we describe the optimal stopping region in $S^M$ implied
by the value function $V_0(\cdot)$, and we present an optimal
sequential decision strategy for our problem. Let us define for
every $N\ge 0$,
\begin{align*}
  \Gamma_N &\triangleq \{\bpi\in S^M\,|\, V_0^N(\bpi)=h(\bpi)\}, &
  \Gamma_N^{(j)} &\triangleq \Gamma_N \cap \{\bpi\in S^M\,|\,
  h(\bpi)=h_j(\bpi)\}, \; j\in \Mh, \\
  \Gamma &\triangleq \{\bpi\in S^M\,|\, V_0(\bpi)=h(\bpi)\}, &
  \Gamma^{(j)} &\triangleq \Gamma \cap \{\bpi\in S^M\,|\,
  h(\bpi)=h_j(\bpi)\}, \; j\in \Mh.
\end{align*}

Theorem \ref{T:sigma-properties} below shows that it is always optimal
to stop and raise an alarm as soon as the posterior probability
process $\Pi$ enters the region $\Gamma$. Intuitively, this follows
from the optimality equation (\ref{E:Dyn-prog-eqn}). At any stage, we
always have two choices: either we stop immediately and raise an alarm
or we wait for at least one more stage and take an additional
observation.  If the posterior probability of all possibilities is
given by the vector $\pi$, then the costs of those competing actions
equal $h(\pi)$ and $c(1-\pi_0)+(\T V_0)(\pi)$, respectively, and it is
always better to take the action that has the smaller expected cost.
The cost of stopping is less (and therefore stopping is optimal) if
$h(\pi) \le c(1-\pi_0)+(\T V_0)(\pi)$, equivalently, if $V_0(\pi) =
h(\pi)$.  Likewise, if at most $N$ stages are left, then stopping is
optimal if $V^N_0(\pi)=h(\pi)$ or $\pi \in \Gamma_N$.

For each $j\in \{0\}\cup \Mh$, let $\e_j\in S^M$ denote the unit
vector consisting of zero in every component except for the $j$th
component, which is equal to one. Note that
$\e_0,\ldots,\e_M$ are the extreme points of the closed
convex set $S^M$, and any vector $\bpi=(\pi_0,\ldots,\pi_M)\in S^M$
can be expressed in terms of $\e_0,\ldots,\e_M$ as $\bpi =
\sum_{j=0}^{M}\pi_j\e_j$.

\begin{theorem}\label{T:Gamma-decreasing-subsets}
  For every $j\in \Mh$, $(\Gamma_N^{(j)})_{N\geq 0}$ is a decreasing
  sequence of non-empty, closed, convex subsets of $S^M$.  Moreover,
  \begin{gather*}
    \Gamma_0^{(j)} \supseteq \Gamma_1^{(j)} \supseteq \cdots \supseteq
    \Gamma^{(j)} \supseteq \left\{\bpi\in S^M
      \,|\,h_j(\bpi)\leq\min\{h(\bpi),c(1-\pi_0)\}\right\} \ni
    \e_j,\\
    \Gamma = \bigcap_{N=1}^{\infty}\Gamma_N =
    \bigcup_{j=1}^{M}\Gamma^{(j)},\quad\text{and}\quad
    \Gamma^{(j)}=\bigcap_{N=1}^{\infty}\Gamma_N^{(j)},\quad
    j\in \Mh.
  \end{gather*}
  Furthermore, $S^M = \Gamma_0 \supseteq \Gamma_1 \supseteq \cdots
  \supseteq \Gamma \supsetneqq \{\e_1,\ldots,\e_M\}$.
\end{theorem}

\begin{lemma}\label{L:gamma-n-V}
  For every $n\geq 0$, we have $\gamma_n =
  -c\sum_{k=0}^{n-1}(1-\Pi_k^{(0)})-V_0(\Pi_n).$
\end{lemma}

\begin{theorem}\label{T:sigma-properties}
  Let $\sigma \triangleq \inf\{n\geq 0 \,|\, \bPi_n \in\Gamma\}$.
  \begin{itemize}
  \item[(a)] The stopped process $\{\gamma_{n \wedge\sigma}, \F_n;
    n\geq 0\}$ is a martingale.
  \item[(b)] The random variable $\sigma$ is an optimal stopping time
    for $V_0$, and
  \item[(c)] $\E\,\sigma<\infty$.
  \end{itemize}
\end{theorem}

Therefore, the pair $(\sigma, d^*)$ is an optimal sequential
decision strategy for \eqref{E:UDef1}, where the optimal stopping
rule $\sigma$ is given by Theorem~\ref{T:sigma-properties}, and,
as in the proof of Lemma~\ref{L:OSP1}, the optimal terminal
decision rule $d^*$ is given by
\begin{align*}
  d^* = j \quad \text{ on the event} \quad \{\sigma=n, \bPi_n\in
  \Gamma^{(j)}\} \quad \text{ for every } n\geq 0.
\end{align*}
Accordingly, the set $\Gamma$ is called the \emph{stopping region}
implied by $V_0(\cdot)$, and
Theorem~\ref{T:Gamma-decreasing-subsets} reveals its basic
structure.  We demonstrate the use of these results in the
numerical examples of Section~\ref{sec:Special-cases}.

Note that we can take a similar approach to prove that the
stopping rules $\sigma_N\triangleq\inf\{n\geq 0\,|\, \bPi_n \in
\Gamma_{N-n}\}, N\geq 0$ are optimal for the truncated problems
$V_0^N(\cdot), N\geq 0$ in (\ref{eq:value-functions}).  Thus, for
each $N\geq 0$, the set $\Gamma_{N}$ is called the stopping region
for $V_0^N(\cdot)$: it is optimal to terminate the experiments in
$\Gamma_N$ if $N$ stages are left before truncation.

\section{Special cases and examples}\label{sec:Special-cases}

In this section we discuss solutions for various special cases of the
general formulation given in Section \ref{sec:Problem-statement}.
First, we show how the traditional problems of Bayesian sequential
change detection and Bayesian sequential multi-hypothesis testing are
formulated via the framework of Section \ref{sec:Problem-statement}.
Then we present numerical examples for the cases $M=2$ and $M=3$. In
particular, we develop a geometrical framework for working with the
sufficient statistic developed in Section \ref{sec:Reformulation} and
the optimal sequential decision strategy developed in Section
\ref{sec:dynamic-programming-solution}. Finally, we solve the special
problem of detection and identification of primary component
failure(s) in a system with suspended animation.

\subsection{A.\ N.\ Shiryaev's sequential change detection problem}

Set $a_{0j}=1$ for $j\in \Mh$ and $a_{ij}=0$ for $i,j\in \Mh$, then
the Bayes risk function \eqref{E:BayesRiskUnderP} becomes
\begin{align*}
  R(\delta) &= c\,\E[(\tau-\theta)^+] + \E[a_{0
    d}\II{\tau<\theta}+a_{\mu d}\II{\theta\leq\tau<\infty}] =
  c\,\E[(\tau-\theta)^+]
  + \E[\II{\tau<\theta}]\\
  &= \P\{\tau<\theta\} + c\,\E[(\tau-\theta)^+].
\end{align*}
This is the Bayes risk studied by Shiryaev
\cite{MR0155708,MR0468067} to solve the sequential change
detection problem.

\subsection{Sequential multi-hypothesis testing}

Set $p_0=1$, then $\theta = 0$ a.s.\ and thus the Bayes risk
function \eqref{E:BayesRiskUnderP} becomes
\begin{align*}
  R(\delta) = c\,\E[(\tau-\theta)^+] + \E[a_{0
    d}\II{\tau<\theta}+a_{\mu d}\II{\theta\leq\tau<\infty}] = \E[c\tau
  + a_{\mu d}\II{\tau<\infty}].
\end{align*}
This gives the sequential multi-hypothesis testing problem studied by
Wald and Wolfowitz \cite{MR0034005}, Arrow, Blackwell, and Girshick
\cite{MR0032173}; see also Blackwell and Girshick \cite{MR597146}.

\subsection{Two alternatives after the change}
\label{sec:two-alternative-case}

In this subsection we consider the special case $M=2$ in which we have
only two possible change distributions, $f_1(\cdot)$ and $f_2(\cdot)$.
We describe a graphical representation of the stopping and
continuation regions for an arbitrary instance of the special case
$M=2$.  Then we use this representation to illustrate geometrical
properties of the optimal method (Section \ref{sec:optimal-soln}) via
model instances for certain choices of the model parameters $p_0$,
$p$, $\nu_1$, $\nu_2$, $f_0(\cdot)$, $f_1(\cdot)$, $f_2(\cdot)$,
$a_{01}$, $a_{02}$, $a_{12}$, $a_{21}$, and $c$.

\begin{figure}[h]
  \ifpdf
  \includegraphics{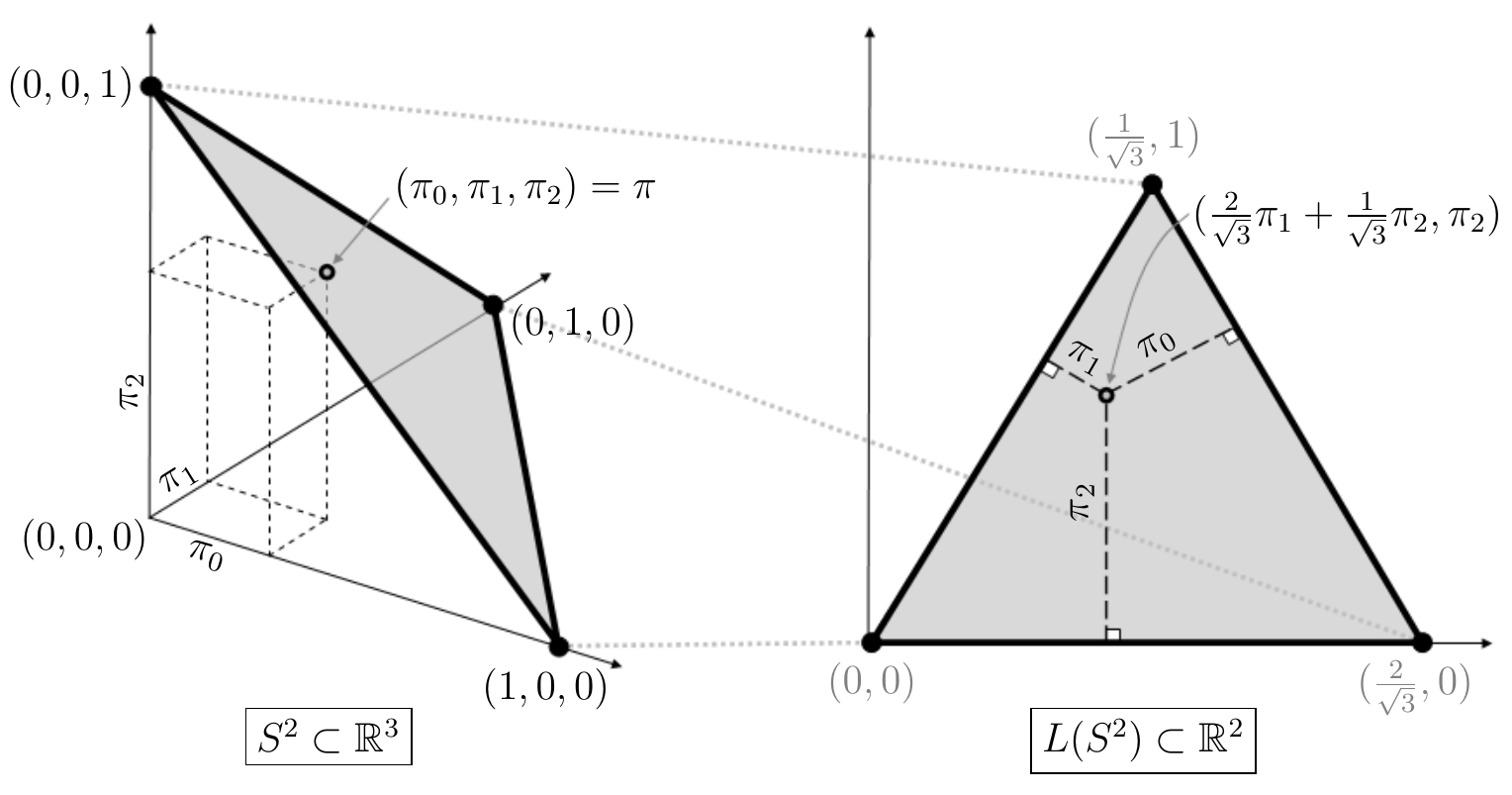}
  \fi
  \caption{Linear mapping $L$ of the standard two-dimensional
    probability simplex $S^2$ from the positive orthant of $\R^3$ into
    the positive quadrant of $\R^2$.} \label{F:S2-to-2D}
\end{figure}

Let the linear mapping $L:\R^3\mapsto\R^2$ be defined by
$L(\pi_0,\pi_1,\pi_2)\triangleq(\tfrac{2}{\sqrt{3}}\pi_1
+\tfrac{1}{\sqrt{3}}\pi_2,\pi_2)$. Since $\pi_0=1-\pi_1-\pi_2$ for
every $\bpi=(\pi_0,\pi_1,\pi_2)\in S^2\subset\R^3$, we can recover the
preimage $\bpi$ of any point $L(\bpi)\in L(S^2)\subset\R^2$.  For
every point $\bpi=(\pi_0,\pi_1,\pi_2)\in S^2$, the coordinate $\pi_i$
is given by the Euclidean distance from the image point $L(\bpi)$ to
the edge of the image triangle $L(S^2)$ that is \emph{opposite} the
image point $L(\e_i)$, for each $i=0,1,2$. For example, the
distance from the image point $L(\bpi)$ to the edge of the image
triangle opposite the lower-left-hand corner $L(1,0,0)=(0,0)$ is the
value of the preimage coordinate $\pi_0$. See Figure~\ref{F:S2-to-2D}.

Therefore, we can work with the mappings $L(\Gamma)$ and
$L(S^2\setminus\Gamma)$ of the stopping region $\Gamma$ and the
continuation region $S^2\setminus\Gamma$, respectively.
Accordingly, we depict the decision region for each instance in
this subsection using the two-dimensional representation as in the
right-hand-side of Figure~\ref{F:S2-to-2D} and we drop the
$L(\cdot)$ notation when labeling various parts of each figure to
emphasize their source in $S^2$.

Each of the examples in this section have the following model
parameters in common:
\begin{gather*}
  p_0=\tfrac{1}{50},\quad p=\tfrac{1}{20},\quad
  \nu_1=\nu_2=\tfrac{1}{2},\\
  f_0=\left(\tfrac{1}{4}, \tfrac{1}{4}, \tfrac{1}{4},
    \tfrac{1}{4}\right),\quad
  f_1=\left(\tfrac{4}{10}, \tfrac{3}{10}, \tfrac{2}{10},
    \tfrac{1}{10}\right),\quad
  f_2=\left(\tfrac{1}{10}, \tfrac{2}{10}, \tfrac{3}{10},
    \tfrac{4}{10}\right).
\end{gather*}
We vary the delay cost and false alarm/isolation costs to illustrate
certain geometrical properties of the continuation and stopping
regions.  See Figures~\ref{F:2D1},~\ref{F:2D2}, and ~\ref{F:2D3}.
\begin{figure}[h]
  \ifpdf
  \includegraphics{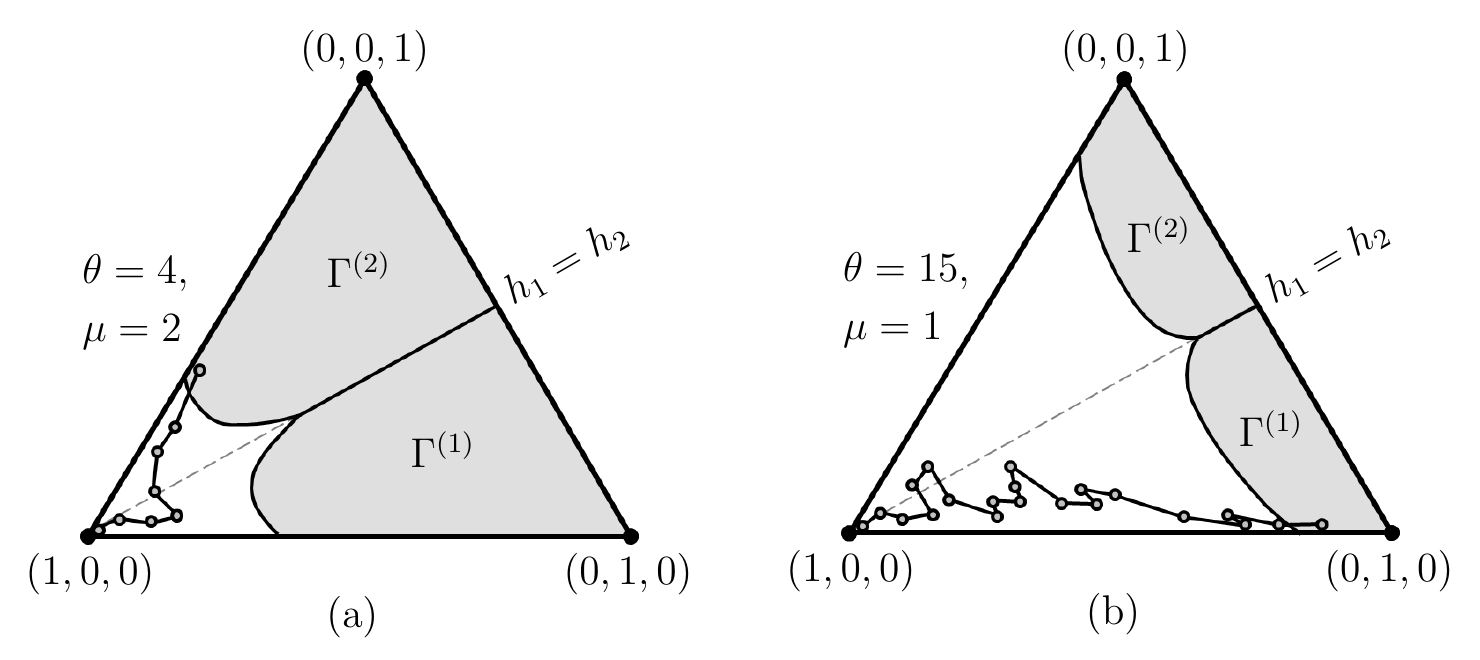}
  \fi
  \caption{Illustration of \emph{connected} stopping regions and the
    effects of variation in the false-alarm costs. (a) and (b):
    $a_{12}=a_{21}=3,\,c=1$. (a): $a_{01}=a_{02}=10$. (b):
    $a_{01}=a_{02}=50$.}
  \label{F:2D1}
\end{figure}
\begin{figure}[h]
  \ifpdf
  \includegraphics{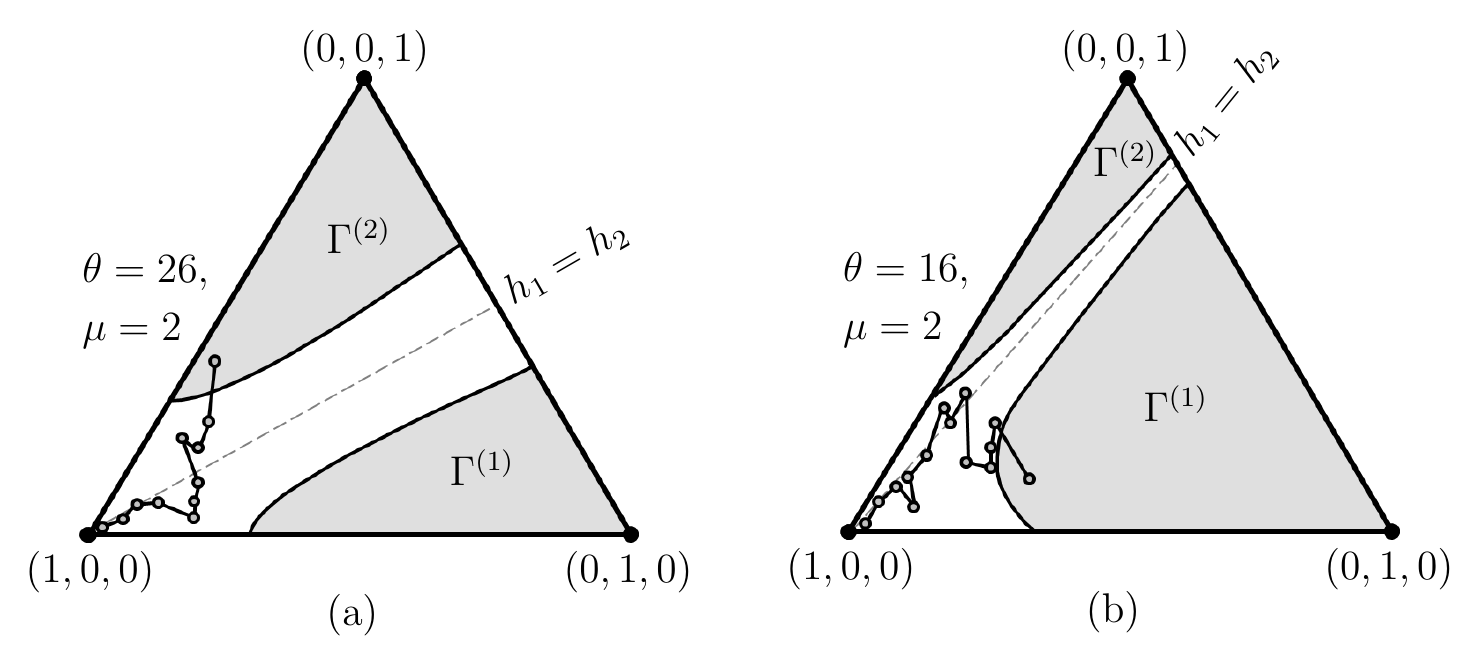}
  \fi
  \caption{Illustration of \emph{disconnected} stopping regions and
    the effects of asymmetric false-isolation costs. (a) and (b):
    $a_{01}=a_{02}=10,\,c=1$. (a): $a_{12}=a_{21}=10$. (b):
    $a_{12}=16,a_{21}=4$.}
  \label{F:2D2}
\end{figure}
\begin{figure}[h]
  \ifpdf
  \includegraphics{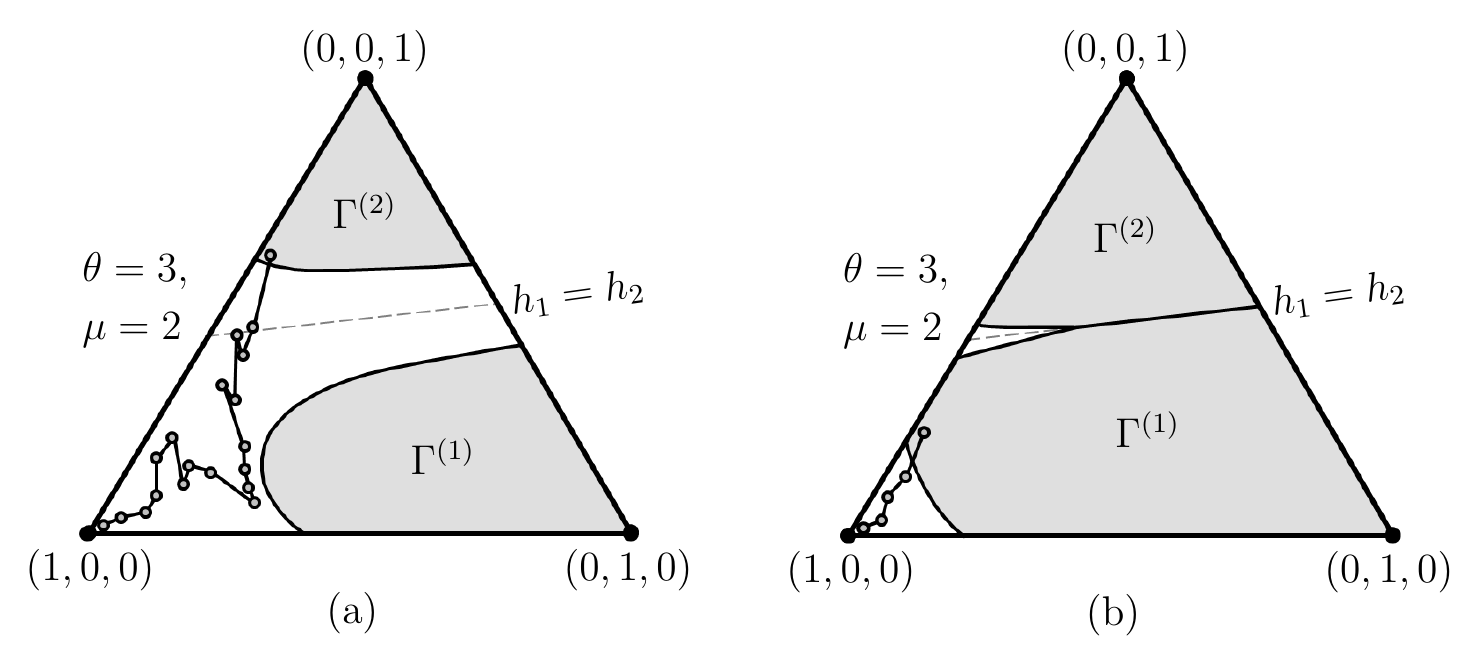}
  \fi
  \caption{Illustration of a \emph{disconnected} continuation region
    and the effects of variation in the delay cost. (a) and (b):
    $a_{01}=14,a_{02}=20,a_{12}=a_{21}=8$. (a): $c=1$.  (b):
    $c=2$.}\label{F:2D3}
\end{figure}

Specifically, these examples show instances in which the $M=2$ convex
subsets comprising the optimal stopping region are connected (Figure
\ref{F:2D1}) and instances in which they are not (Figures \ref{F:2D2}
and \ref{F:2D3}(a)). Figure \ref{F:2D3}(b) shows an instance in which
the continuation region is disconnected.

Each of the figures in this section have certain features in
common. On each subfigure there is a dashed line representing
those states $\bpi\in S^2$ at which $h_1(\bpi)=h_2(\bpi)$.  Also,
each subfigure shows a sample path of $(\bPi_n)_{n=0}^{\sigma}$
and the realizations of $\theta$ and $\mu$ for the sample.  The
shaded area, including its solid boundary, represents the optimal
stopping region, while the unshaded area represents the
continuation region.

An implementation of the optimal strategy as described in Section
\ref{sec:optimal-soln} is as follows: Initialize the statistic
$\bPi=(\bPi_n)_{n\geq 0}$ by setting
$\bPi_0=(1-p_0,p_0\nu_1,p_0\nu_2)$ as in part (c) of Proposition
\ref{P:PiProperties}. Use the dynamics of \eqref{E:Pi-Dynamics} to
update the statistic $\bPi_n$ as each observation $X_n$ is
realized. Stop taking observations when the statistic $\bPi_n$
enters the stopping region $\Gamma=\Gamma^{(1)}\cup\Gamma^{(2)}$
for the first time, possibly before the first observation is taken
(i.e., $n=0$). The optimal terminal decision is based upon whether
the statistic $\bPi_n$ is in $\Gamma^{(1)}$ or $\Gamma^{(2)}$ upon
stopping. Each of the sample paths in
Figures~\ref{F:2D1},~\ref{F:2D2},~and~\ref{F:2D3} were generated
via this algorithm.  As Figure~\ref{F:2D1} shows, the sets
$\Gamma^{(1)}$ and $\Gamma^{(2)}$ can intersect on their
boundaries and so it is possible to stop in their intersection.
In this case, either of the decisions $d=1$ or $d=2$ is optimal.

We use value iteration of the optimality
equation~\eqref{E:Dyn-prog-eqn} over a fine discretization of $S^2$ to
compute $V_0(\cdot)$ and generate the decision region for each
subfigure.  Because in the expression $V_0(\pi) =
\min\{h(\pi),c(1-\pi_0)+(\T V_0)(\pi)\}$ the value $V_0(\pi)$ for any
fixed initial condition $\Pi_0 = \pi$ on the left depends on the
\emph{entire} function $V_0(\cdot)$ on $S^M$ on the right, we have to
calculate $V_0(\cdot)$ (or approximate it by $V^N(\cdot)$) on the
entire space $S^M$. The resulting discretized decision region is
mapped into the plane via $L$.  

See Bertsekas \cite[Chapter 3]{MR2182753} for techniques of computing
the value function via the optimality equation such as value
iteration.  Solving the optimality equation by discretizing
high-dimensional state-space may not be the best option.  Monte Carlo
methods based on regression models for the value function seem to
scale better as the dimension of the state-space increases; see, for
example, Longstaff and Schwartz \cite{LS01}, Tsitsiklis and van Roy
\cite{tsitsiklis01regression}, Glasserman \cite[Chapter 8]{MR1999614}
for details.

\subsection{Three alternatives after the change}
\label{sec:three-alternatives}
    \begin{figure}[ht]
        \ifpdf
        \includegraphics[width=0.7\textwidth]{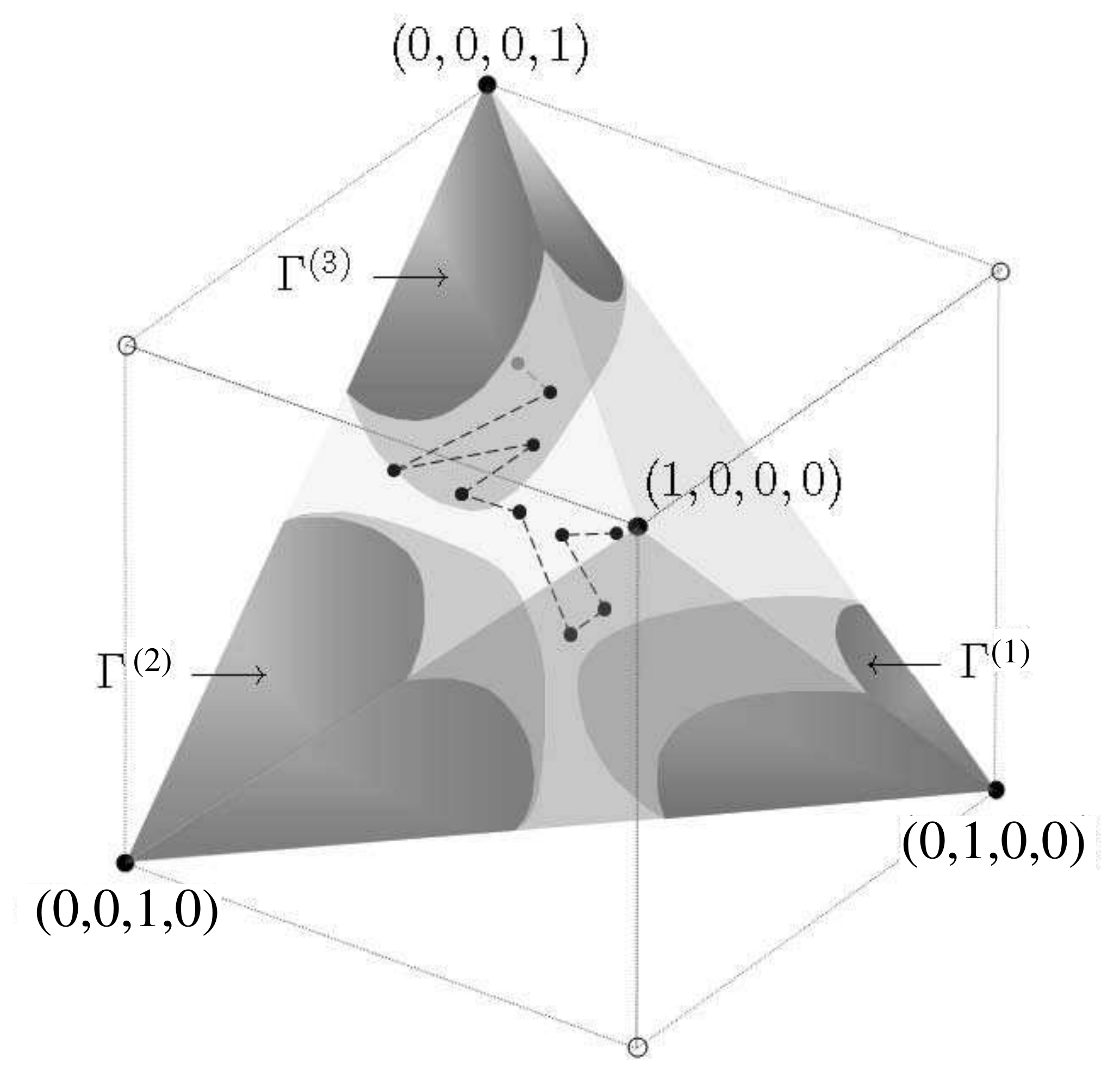}
        \fi
        \caption{Illustration of the mapped decision region for an
          instance of the special case $M=3$; see also Figure
          \ref{fig:param3D} below.  A sample path of the process
          $\bPi$ is shown in which $\theta=6$ and $\mu=3$}
        \label{F:3D}
    \end{figure}
In this subsection we consider the special case $M=3$ in which we
have three possible change distributions, $f_1(\cdot)$,
$f_2(\cdot)$, and $f_3(\cdot)$. Here, the continuation and
stopping regions are subsets of $S^3\subset\R^4$. Similar to the
two-alternatives case, we introduce the mapping of
$S^3\subset\R^4$ into $\R^3$ via
    \begin{align*}
        (\pi_0,\pi_1,\pi_2,\pi_3)\mapsto
        \left(\sqrt{\tfrac{3}{2}}\pi_1
        +\tfrac{1}{2}\sqrt{\tfrac{3}{2}}\pi_2
        +\tfrac{1}{2}\sqrt{\tfrac{3}{2}}\pi_3,
        \tfrac{3}{2}\sqrt{\tfrac{1}{2}}\pi_2 +
        \tfrac{1}{2}\sqrt{\tfrac{1}{2}}\pi_3, \pi_3\right).
    \end{align*}
Then we use this representation---actually a rotation of it---to
illustrate in Figure~\ref{F:3D} an instance with the following
model parameters:
    \begin{gather*}
        p_0=\tfrac{1}{50},\quad p=\tfrac{1}{20},\quad
        \nu_1=\nu_2=\nu_3=\tfrac{1}{3}, \\
        f_0=\left(\tfrac{1}{4}, \tfrac{1}{4}, \tfrac{1}{4},
        \tfrac{1}{4}\right),\quad
        f_1=\left(\tfrac{4}{10}, \tfrac{3}{10}, \tfrac{2}{10},
        \tfrac{1}{10}\right),\quad
        f_2=\left(\tfrac{1}{10}, \tfrac{2}{10}, \tfrac{3}{10},
        \tfrac{4}{10}\right),\quad
        f_3=\left(\tfrac{3}{10}, \tfrac{2}{10}, \tfrac{2}{10},
        \tfrac{3}{10}\right) \\
        c=1,\quad a_{0j}=40,\quad a_{ij}=20,\quad i,j=1,2,3.
    \end{gather*}

Note that Figure~\ref{F:3D} can be interpreted in a manner similar
to the figures of the previous subsection.  In this case, for
every point $\bpi=(\pi_0,\pi_1,\pi_2,\pi_3)\in S^3$, the
coordinate $\pi_i$ is given by the (Euclidean) distance from the
image point $L(\bpi)$ to the face of the image tetrahedron
$L(S^3)$ that is opposite the image corner $L(\e_i)$, for
each $i=0,1,2,3$.

\subsection{Detection and identification of component failure(s)
in a system with suspended animation}
\label{sec:suspended-animation}

Consider a system consisting initially of two working concealed
components (labeled 1 and 2) such that upon the failure of either
component, the system goes into a state of suspended animation.
That is, while both components are still working normally,
observations of output of the system have density $f_0(\cdot)$,
but upon failure of either component the density of observations
changes thereafter (until an alarm is raised) to one of two
alternatives: if component $2$ fails before component $1$, then
post-failure observations have density $f_2(\cdot)$, otherwise
they have density $f_1(\cdot)$. The problem is to detect quickly
when there has been a component failure \emph{and} to identify
accurately which component has actually failed based only on
sequential observations of output of the system.

Let the random variables
\begin{align*}
    \theta := \theta_1\wedge\theta_2 = \min\{\theta_1,\theta_2\}\quad\text{and}\quad\mu :=
    \left\{\begin{aligned}&1 &&\text{if }\theta_1\leq\theta_2\\&2 &&\text{if }\theta_1>\theta_2\end{aligned}\right.
\end{align*}
be respectively the time of failure of the first failed component
of the system and the corresponding index of this component, where
the failure time $\theta_i$ of the $i$th component is a random
variable having a geometric distribution with failure probability
$p_i$, $i=1,2$. It can be shown easily that when the disorder
times $\theta_1$ and $\theta_2$ are independent, the random
variable $\theta$ has a geometric distribution with failure
probability $p:=p_1+p_2-p_1p_2$ (or equivalently, $\theta$ has a
zero-modified geometric distribution with parameters $p_0=0$ and
$p$) and that it is independent of the random variable $\mu$,
which has distribution $\nu_1=p_1/p$ and $\nu_2=1-\nu_1$. So
although the failure type (i.e., which component has failed) is a
function of the failure times of each component, it turns out that
this problem fits properly within the Bayesian sequential change
diagnosis framework.

This problem can be extended naturally to several components and
solved via the technology of Sections~\ref{sec:Reformulation} and
\ref{sec:dynamic-programming-solution}.  In fact, it can be configured
for a variety of scenarios.  For example, series-connected components
where malfunction of one component suspends immediately the operation
of all the remaining components can appear in various electronic
relays and multicomponent electronic devices which have fuses to
protect the system from the misbehavior of one of its components.
Since the system may react differently to diagnostics run by the
operators, post-malfunction behavior can differ according to the
underlying cause of the malfunction.  See Barlow~\cite[Section
8.4]{MR1621421} for background on series systems with suspended
animation.  Consider also a manufacturing process where we perform a
quality test on the final output produced from several processing
components.  If a component is highly reliable then a geometric
distribution with a low failure rate can be a reasonable choice for
the lifetime of the component.  Moreover, since the typical duration
between successive component failures widens over time we can often
treat the remaining components as if they enter a state of suspended
animation under certain cost structures.  That is, we can expect the
remaining components to outlive the alarm.  For example, suppose that
two independent geometric random variables have expected lifetime of
$1000$ each.  Then the first failure will occur at about time $500$ on
average, while the second failure will take an additional $1000$
periods on average to occur.  As illustrated in
Figures~\ref{F:2D1}~and~\ref{F:2D3}, respectively, lower false-alarm
costs promote raising the alarm earlier, while a higher delay cost
discourages waiting for more than relatively few additional periods to
raise the alarm.

Specifically, suppose that in a ``black box" there are $K$
components whose lifetimes are independent and geometrically
distributed.  Observations have initially distribution
$f_0(\cdot)$ while the system is working, but upon failure of a
single component (or simultaneous failure of multiple components),
the remaining components enter a state of suspended animation, and
the post-failure distribution of observations is determined by the
failed component(s).  We want to detect the time when at least one
of them fails as soon as possible.  Moreover, when we raise an
alarm we would like to be able to make as accurately as possible
diagnoses such as (1) \emph{how many} of the components have
actually failed, and (2) \emph{which} ones.

Again, let the failure time $\theta_k$ of the $k$th component be a
random variable having a geometric distribution with failure
probability $p_k$, $k\in\mathcal{K}:=\{1,2,\ldots,K\}$, and define
\begin{equation*}
    \theta := \theta_1\wedge\theta_2\wedge\cdots\wedge\theta_K =
    \min_{k\in\mathcal{K}}\theta_k
\end{equation*}
as the time when at least one of the $K$ components fails.  Let
the mapping $\varphi:2^{\mathcal{K}}\mapsto\{0,1,\ldots\}$ be a
nonnegative-integer-valued measure on the discrete
$\sigma$-algebra $2^{\mathcal{K}}$ of the set
$\mathcal{K}=\{1,2,\ldots,K\}$ of component indices, and define
the random variable
\begin{equation*}
    \mu := \varphi(\{k\in\mathcal{K}\,|\,\theta=\theta_k\})
\end{equation*}
as an index function on the set of indices of the failed
components. When the random variables $\theta_1, \ldots, \theta_K$
are independent, it can be shown that the random variable $\theta$
has a geometric distribution with failure probability
\begin{equation*}
    p:=1-\prod_{i\in\mathcal{K}}(1-p_i)
\end{equation*}
and that it is independent of the random variable $\mu$, which has
distribution
\begin{equation*}
    \nu_k :=\frac{1}{p}\sum_{A\in\varphi^{-1}(k)}\prod_{i\in A}p_i
    \prod_{j\in\mathcal{K}\setminus A} (1-p_j),\quad
    k\in\Mh:=\{1,2,\ldots,M:=\varphi(\mathcal{K})\}.
\end{equation*}

So, the preceding example of two components corresponds to the
special case where $K=2$ and $\varphi(A)=\min A$ for $A\in
\{\{1\},\{1,2\},\{2\}\}$. We can handle the other two
aforementioned objectives as follows:

\begin{enumerate}
\item[(1)] Let $\varphi(A)=|A|,A\in 2^{\mathcal{K}}$.  Then the random
variable $\mu$ represents how many components fail.

\item[(2)] Let $\varphi(A)=\sum_{i\in A}2^{i-1},A\in 2^{\mathcal{K}}$.
Then the mapping $\varphi$ is one-to-one, the random variable
$\mu$ takes values in $1,2,\ldots,2^K-1$, and the set
$\varphi^{-1}(\mu)$ consists of the indices of the components
which fail; i.e., the random variable $\mu$ identifies uniquely
which components fail.
\end{enumerate}

%

%

\section{On the computer implementation of the change-diagnosis
  algorithm.}
\label{sec:computer-implementation}
Updating posterior probability process $\Pi$ online with a computer by
using the recursive equations in (\ref{eq:D-mappings}) and
(\ref{E:Pi-Dynamics}) is fast. However, programming a computer to
check online whether this process has just entered the optimal
stopping region is a challenging task. This is especially so because
(i) the critical boundaries of stopping regions do not have known
closed-form expressions, and (ii) extensive online computations to
determine if one of these boundaries is crossed can take excessive
time and defeat the purpose of quickest change detection.  Here we
outline an implementation strategy that should perform well in
general.

The strategy is based on sparse offline representations of critical
boundaries between stopping and continuation regions. Suppose that the
posterior probability process $\Pi$ has just been updated to some
$\pi=(\pi_0,\pi_1,\ldots,\pi_M) \in S^M$.  An alarm has to be
raised if and only if $\pi \in \Gamma \equiv \Gamma^{(1)} \cup \cdots
\cup \Gamma^{(M)}$. Checking $\pi \in \Gamma^{(i)}$ for every
$i=1,\ldots,M$ (in the worst case) is, however, unnecessary because
\begin{align*}
  \pi \in \Gamma \quad \Longleftrightarrow \quad \pi \in \Gamma_i \qquad
  \text{if} \quad h_i(\pi) = h(\pi) \equiv \min_{1\le j \le M}
  h_j(\pi).
\end{align*}
In other words, one should
\begin{enumerate}
\item find $i = \arg\min_{1\le j \le M} \;h_j(\pi)$ first, and
\item raise an alarm and declare that a change of type $i$ has
  happened if $\pi \in \Gamma^{(i)}$, or
\item wait for at least one more period before raising any alarm
  otherwise.
\end{enumerate}
Let us suppose that $i = \arg\min_{1\le j \le M} \;h_j(\pi)$. Checking
if $\pi \in \Gamma^{(i)}$ will be fast if both $\pi$ and
$\Gamma^{(i)}$ are represented in terms of polar coordinates, set up
locally relative to the corner of $S^M$ confined in the convex set
$\Gamma^{(i)}$.

\begin{figure}
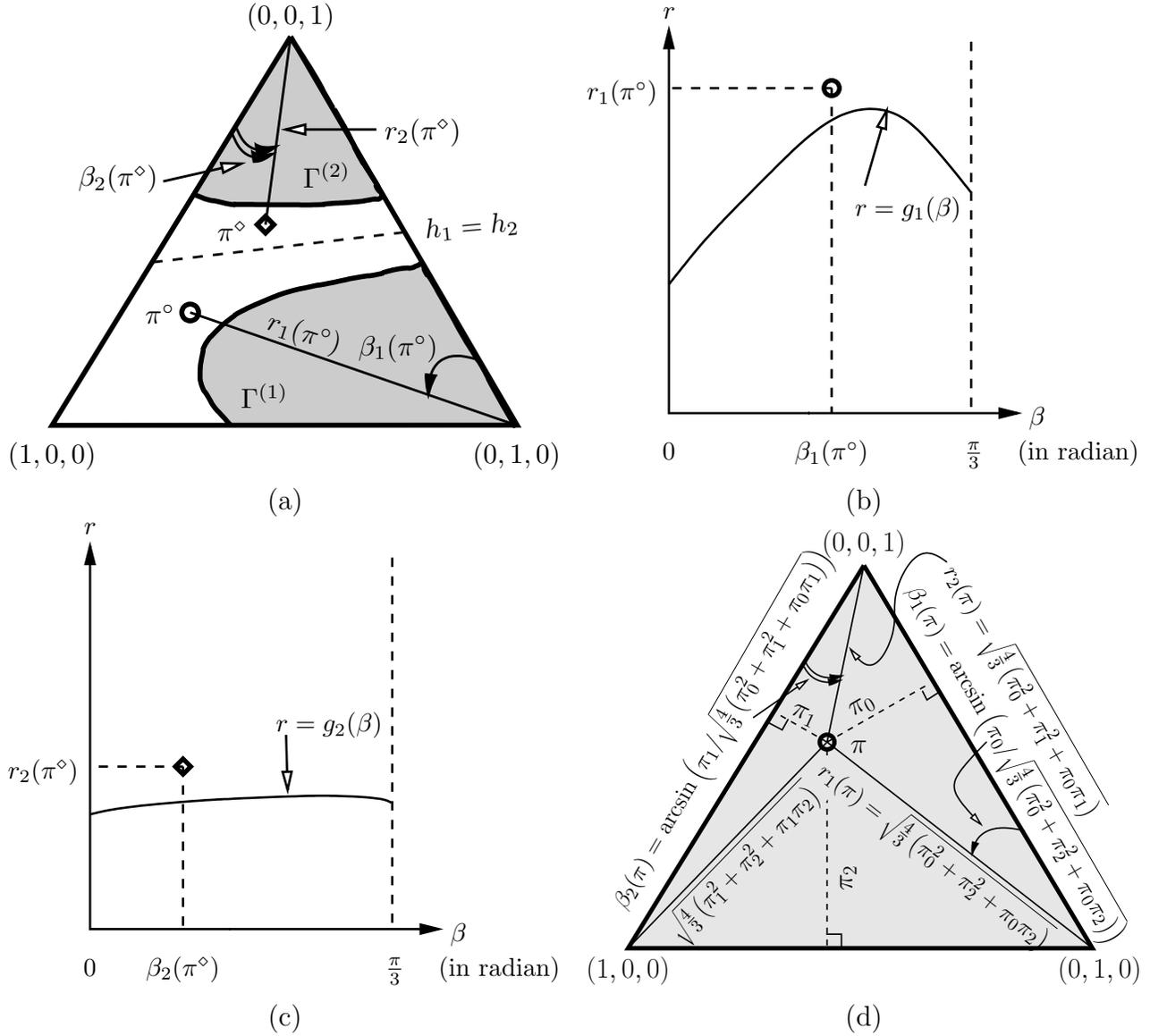

\begin{tabular}{cc}
\includegraphics[width=0.5\textwidth]{img2D3modified.mps}
&
\includegraphics[width=0.5\textwidth]{bndr2D1.mps} \\
(a) & (b) \\
\includegraphics[width=0.5\textwidth]{bndr2D2.mps} 
&
\includegraphics[width=0.5\textwidth]{param2D.mps} \\
(c) & (d) \\
\end{tabular}
\caption{For the sample problem displayed in Figure \ref{F:2D3}(a)
  ($M=2$), optimal stopping regions and local polar coordinate systems
  are shown in (a). The critical boundaries of the stopping regions
  $\Gamma^{(1)}$ and $\Gamma^{(2)}$ are expressed in terms of local
  polar coordinates in (b) and (c), respectively. In (d), polar
  coordinates of $\pi$ are stated in terms of its Cartesian
  coordinates. As in Section \ref{sec:two-alternative-case}, we drop
  $L$ from $L(\pi^{\circ})$, $L(\Gamma^{(1)})$, $L(1,0,0)$, etc.  and
  simply write $\pi^{\circ}$, $\Gamma^{(1)}$, $(1,0,0)$. In (a)
  $h_1(\pi^{\circ}) \le h_2(\pi^{\circ})$ and $h_1(\pi^{\diamond}) \ge
  h_2(\pi^{\diamond})$.}
\label{fig:param2D}
\end{figure}

To illustrate the ideas with simple pictures, we will focus on the
case that there are $M=2$ alternatives after the change; see Figure
\ref{fig:param2D}. If $\pi = \pi^{\circ}$ (respectively,
$\pi=\pi^{\diamond}$) as in Figure \ref{fig:param2D}(a), then $h_1(\pi) \le
h_2(\pi)$ and $i=1$ (respectively, $h_1(\pi) \ge h_2(\pi)$ and $i=2$).
In either case, $\pi$ can be identified relative to each corner in
terms of (i) the Euclidean distance to that corner (denoted by
$r_j(\pi)$, $j=0,1,2$) and (ii) one arbitrary but fixed angle (say, by
$\beta_j(\pi)$, $j=0,1,2,3$ indicated on Figure \ref{fig:param2D}(a)) between
the line connecting $\pi$ and the corner and the rays forming the same
corner. Every point on the critical boundary of the stopping region
$\Gamma^{(j)}$, $j=1,2$ admits the same representation. Let us express
by $r= g_j(\beta)$ the critical boundary of the stopping region
$\Gamma^{(j)}$ in terms of the polar coordinates $(\beta,r)$ measured
locally with respect to the corner of the simplex confined in
$\Gamma^{(j)}$ for $j=1,2$. Then $\pi^{\circ} \in \Gamma$ if and only
if $r_1(\pi^{\circ}) \le g_1(\beta_1(\pi^{\circ}))$, and
$\pi^{\diamond} \in \Gamma$ if and only if $r_2(\pi^{\diamond}) \le
g_2(\beta_2(\pi^{\diamond}))$; see Figures \ref{fig:param2D}(b) and
\ref{fig:param2D}(c), respectively.

The plan outlined above works well (i) if the local polar coordinates
of $\pi$ can be identified online quickly, (ii) if the local
representations $g_j(\cdot)$, $j=1,2$ of the critical boundaries can be
stored efficiently to the computer memory, and (iii) if from there
they can be retrieved and evaluated fast on demand. Next we will
explain how these requirements can be achieved.

Recall from Section \ref{sec:two-alternative-case} and Figure
\ref{F:S2-to-2D} that $\pi \in S^2\subset \R^3$ is embedded into the
equilateral triangle $L(S^2)\subset \R^2$ by means of a linear map
$\pi \mapsto L(\pi)$.  In this natural representation of posterior
distributions, $\pi=(\pi_0,\pi_1,\pi_2)$ is mapped to the point
$L(\pi)=\left(\frac{2}{\sqrt{3}}\pi_1+\frac{1}{\sqrt{3}}\pi_2,\pi_2\right)$,
whose Euclidean distance to the images $L(1,0,0)=(0,0)$,
$L(0,1,0)=(\frac{2}{\sqrt{3}},0)$, $L(0,0,1)=(\frac{1}{\sqrt{3}},1)$
(corners of the equilateral triangle $L(S^2)$) of $(1,0,0)$,
$(0,1,0)$, $(0,0,1)$ are
\begin{align*}
  r_0(\pi) &\triangleq \|L(1,0,0) - L(\pi)\| =
  \sqrt{\frac{4}{3}\left(\pi^2_1+
      \pi^2_2+ \pi_1 \pi_2\right)}, \\
  r_1(\pi) &\triangleq \|L(0,1,0) - L(\pi)\| =
  \sqrt{\frac{4}{3}\left(\pi^2_0+
      \pi^2_2+ \pi_0 \pi_2\right)}, \\
  r_2(\pi) &\triangleq \|L(0,0,1) - L(\pi)\| =
  \sqrt{\frac{4}{3}\left(\pi^2_0+ \pi^2_1+ \pi_0 \pi_1\right)},
\end{align*}
respectively; in a more compact way,
\begin{align}
\label{eq:distance-in-2D}
  r_i(\pi) = \sqrt{\frac{4}{3} \left(- \pi_i + \sum_{0\le j\le k \le 2}
      \pi_j\pi_k  \right) }, \quad i=0,1,2;
\end{align}
see Figure \ref{fig:param2D}(d). Because the Euclidean distance of
$L(\pi)$ to the edges opposite to the corners $L(1,0,0)$, $L(0,1,0)$,
$L(0,0,1)$ are $\pi_0$, $\pi_1$, and $\pi_2$, respectively, the angles
identified in Figure \ref{fig:param2D}(a) can also be calculated
easily by
\begin{align*}
  \beta_1(\pi) = \arcsin \frac{\pi_0}{\sqrt{\frac{4}{3}\left(\pi^2_0+
        \pi^2_2+ \pi_0 \pi_2\right)}} \quad \text{and} \quad
  \beta_2(\pi) = \arcsin \frac{\pi_1}{\sqrt{\frac{4}{3}\left(\pi^2_0+
        \pi^2_1+ \pi_0 \pi_2\right)}}; 
\end{align*}
or more compactly by
\begin{align}
\label{eq:angle-in-2D}
  \beta_i(\pi) = \frac{\pi_{i+2 \bmod{3}}}{r_i(\pi)} = \frac{\pi_{i+2
      \bmod{3}}}{\sqrt{\frac{4}{3} \left(- \pi_i + \sum_{0\le j\le k \le
          2} \pi_j\pi_k \right) }}, \qquad i=0,1,2.
\end{align}
Recall that at any $\pi \in S^2$ one has to calculate $\beta_i(\pi)$
and $r_i(\pi)$ only for $i = \arg\min_{1\le j \le 3 } h_j(\pi)$ and
check if $r_i(\pi) \le g_i (\beta_i(\pi))$ before raising an alarm.

Unfortunately, an exact/closed-form representation $r=g_j(\beta)$ of
the critical boundary of the stopping region $\Gamma^{(j)}$, $j=1,2$
in terms of the local polar coordinates $(\beta, r)$ relative to the
corner confined in $\Gamma^{(j)}$ will almost never be available.
Instead, only noisy observations (due to the discretization of the
state-space $S^2$ and termination of the value iteration at some
finite stage) of that relation can be obtained from the pairs
$(\beta_j(\pi),r_j(\pi))$ for every grid-point $\pi$ on the
(approximate) critical boundary of $\Gamma^{(j)}$ for every $j=1,2$.
Interpolation between those points will certainly give an
approximation for $r=g_j(\beta)$ for $j=1,2$, but this may waste a lot
storage space and computational time during online evaluations,
especially when the grid on $S^2$ is fine.  Instead, one can use some
statistical smoothing technique to compress the data with minimum loss
of information.

Let us suppose that $N$ observations $(\beta^{(k)},r^{(k)})$,
$k=1,\ldots,N$ follow the model $r^{(k)} = g_1(\beta^{(k)}) +
\varepsilon^{(k)}$ for every $k=1,\ldots,N$ and that
$\varepsilon^{(k)}$, $k=1,\ldots,N$ are i.i.d.\ random variables with
zero mean and some finite common variance.  Because $\Gamma^{(1)}$ is
convex, the function $\beta\mapsto g_1(\beta)$ is concave, namely,
fairly smooth. It may be plausible to approximate it by a cubic spline
(twice continuously differentiable piecewise cubic polynomial). The
unique curve $\beta \mapsto \widehat{g}_1(\beta)$ that has the minimum
penalized sum of squared errors
\begin{align}
  \label{eq:penalized-sum-of-squares}
  S_{\lambda}(\widehat{g}_1) \triangleq \sum^N_{k=1}
  \left[r^{(k)}-\widehat{g}_1(\beta^{(k)})\right]^2 + \lambda \int_{\R}
  \left[\widehat{g}_1''(\beta) \right]^2 d\beta,
\end{align}
for any arbitrary but fixed smoothing parameter $\lambda>0$, among all
twice-differentiable curves is known to exist and belong to the family
of cubic splines whose break-points are at $\beta_1,\ldots,\beta_N$;
see, for example, de Boor \cite{MR1900298}, Green and Silverman
\cite{MR1270012}, Ramsay and Silverman \cite{MR2168993}.  This
optimality property and the ability to control the smoothness
continuously through $\lambda$ make cubic splines an attractive
candidate for an approximate $g_1(\cdot)$.  If the variation of the
original curve $\beta\mapsto g_1(\beta)$ is moderate, then the number
of break-points $0\le K \le N$ can be taken significantly less than the
number of measurements $N$, and there are $O(K)$-algorithms that find
the cubic spline minimizing (\ref{eq:penalized-sum-of-squares}) with
the given $K$ break-points; see, for example, Green and Silverman
\cite[Section 2.3.3]{MR1270012} for Reinsch algorithm.  Other
algorithms represent the solution as a basis-function expansion
\begin{align*}
  \widehat{g}_1 (\beta) = \sum^{K+3}_{j=1} c_j \Phi_j(\beta)
\end{align*}
in terms of $K+3$ spline basis functions $\Phi_1,\ldots, \Phi_{K+3}$,
and solve the minimization problem in
(\ref{eq:penalized-sum-of-squares}) by finding the coefficients
$c_1,\ldots,c_{K+3}$ using multiple-regression; see Green and
Silverman \cite[Section 3.6]{MR1270012}, Ramsay and Silverman
\cite[Section 3.5 and Chapter 5]{MR2168993}. Thus, the approximation
$\widehat{g}_1(\cdot)$ of $g_1(\cdot)$ can be stored to the computer
memory for online use of the change-diagnosis algorithm by means of
only $K+3$ numbers $c_1,\ldots,c_{K+3}$. The basis functions
$\Phi_1,\Phi_2,\ldots$ are cubic splines with compact support and can
be stored easily and evaluated fast online.

All of the above ideas apply without affecting significantly the
online performance of the diagnosis algorithm when the number of
alternatives $M$ after change is larger than two. For example, if
$M=3$, then $S^3\subset \R^4$ is embedded into a tetrahedron $L(S^3)
\subseteq \R^3$ by a linear map $\pi \mapsto L(\pi)$ defined in
Section \ref{sec:three-alternatives}. The Euclidean distance of
$L(\pi)$ to the images $L(1,0,0,0)$, $L(0,1,0,0)$, $L(0,0,1,0)$,
$L(0,0,0,1)$ of $(1,0,0,0)$, $(0,1,0,0)$, $(0,0,1,0)$, $(0,0,0,1)$ are
given by
\begin{align*}
  r_0(\pi) &\triangleq \|L(1,0,0,0) - L(\pi)\| =
  \sqrt{\frac{3}{2}\left(\pi^2_1+\pi^2_2+ \pi^2_3+ \pi_1 \pi_2 +
      \pi_1\pi_3 +\pi_2\pi_3\right)}, \\
  r_1(\pi) &\triangleq \|L(0,1,0,0) - L(\pi)\| =
  \sqrt{\frac{3}{2}\left(\pi^2_0+\pi^2_2+ \pi^2_3+ \pi_0 \pi_2 +
      \pi_0\pi_3 + \pi_2\pi_3\right)}, \\
  r_2(\pi) &\triangleq \|L(0,0,1,0) - L(\pi)\| =
  \sqrt{\frac{3}{2}\left(\pi^2_0+ \pi^2_1+ \pi^2_3 + \pi_0\pi_1 +
      \pi_0\pi_3 + \pi_1\pi_3 \right) }, \\
  r_3(\pi) &\triangleq \|L(0,0,0,1) - L(\pi)\| =
  \sqrt{\frac{3}{2}\left(\pi^2_0+ \pi^2_1+ \pi^2_2 + \pi_0\pi_1 +
      \pi_0\pi_2 + \pi_1\pi_2 \right) },
\end{align*}
respectively; or more compactly
\begin{align}
\label{eq:ditance-in-3D}
  r_i(\pi) = \sqrt{\frac{3}{2} \left(- \pi_i + \sum_{0\le j\le k \le 3}
      \pi_j\pi_k  \right) }, \quad i=0,1,2,3;
\end{align}
see Figure \ref{fig:param3D}. Because the Euclidean distance of
$L(\pi)$ to the faces of the tetrahedron opposite to the corners
$L(1,0,0,0)$, $L(0,1,0,0)$, $L(0,0,1,0)$, $L(0,0,0,1)$ are $\pi_0$,
$\pi_1$, $\pi_2$, and $\pi_3$, respectively, the distance $r_i(\pi)$
and two arbitrary but fixed angles, $\beta_i(\pi)=(\beta_{i1}(\pi),
\beta_{i2}(\pi))$, out of three angles defined by
\begin{align}
\label{eq:angle-in-3D}
  \arcsin \frac{\pi_j}{r_i(\pi)} = \arcsin
  \frac{\pi_j}{\sqrt{\frac{3}{2} \left(- \pi_i + \sum_{0\le k \le \ell
          \le 3} \pi_k\pi_{\ell} \right) }}, \quad 0\le j\le 3,\;
  j\not= i
\end{align}
form the local polar coordinates $(\beta_i(\pi),r_i(\pi))$ with
respect to the corner of the simplex confined in $\Gamma^{(i)}$, $0\le
i \le 3$ and determine $L(\pi)$ uniquely.

\begin{figure}
  \includegraphics[width=\textwidth]{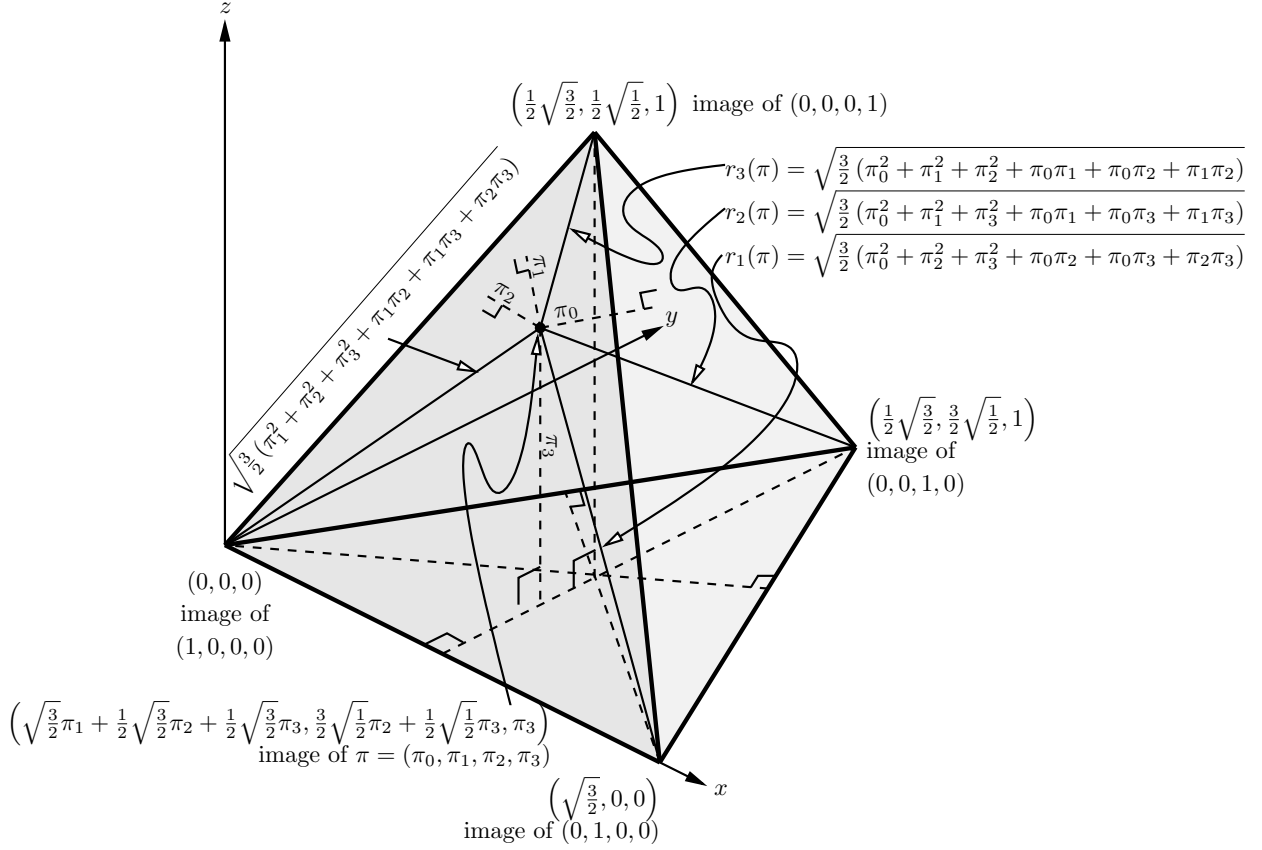}
  \caption{Polar coordinates of $\pi$ (after transformation by $L$;
    see Section \ref{sec:three-alternatives}) in terms of its
    Cartesian coordinates.}
\label{fig:param3D}
\end{figure}

The critical boundary between stopping region $\Gamma^{(i)}$, $1\le i
\le 3$ and the continuation region can be represented by some concave
surface $r = g_i(\beta)$ in terms of the same local polar coordinate
system $(\beta, r)$ just defined above in the vicinity of
$\Gamma^{(i)}$, where $\beta=(\beta_1,\beta_2)$ is now a vector.  If
$(\beta^{(k)}, r^{(k)})$, $k=1,\ldots,N$ are the pairs
$(\beta_1(\pi),r_1(\pi))$ evaluated at grid-points $\pi$ on the
approximate boundary of $\Gamma^{(1)}$, then one can fit a \emph{thin
  plane spline} $\widehat{g}_1(\cdot)$, which is twice continuously
differentiable and minimizes the penalized sum of squared errors
\begin{align*}
  S_{\lambda}(\widehat{g}_1) \triangleq \sum^N_{k=1}
  \left[r^{(k)}-\widehat{g}_1(\beta^{(k)}) \right]^2 + \lambda
  \sum_{1\le i,j \le 2} \iint_{\R^2}   \left(\frac{\partial^2
        \widehat{g}_1}{\partial\beta_i\partial
        \beta_j}\right)^2 (\beta_1,\beta_2) d\beta_1\,d\beta_2
\end{align*}
among all twice-differentiable curves on $\R^2$ for every arbitrary
but fixed smoothing parameter $\lambda>0$. As before, $\widehat{g}_1
(\beta) = \sum^{K+3}_{j=1} c_j \Phi_j(\beta)$ admits a basis-function
expansion, and the coefficients $c_1,\ldots,c_{K+3}$ can be found by
using multiple-regression and stored in the computer memory for the
online use of change-diagnosis algorithms. See Green and Silverman
\cite[Chapter 7]{MR1270012} for statistical data smoothing in three
and higher dimensional Euclidean spaces by using thin plate splines.
The similarity of the local polar coordinates
(\ref{eq:distance-in-2D}), (\ref{eq:angle-in-2D}) for $M=2$ and
(\ref{eq:ditance-in-3D}), (\ref{eq:angle-in-3D}) for $M=3$ suggest
that for general $M\ge 2$ and for a suitable constant $c_M>0$
\begin{align*}
  r_i(\pi) = \sqrt{c_M \left( -\pi_i + \sum_{0\le j \le k \le M}
      \pi_j \pi_k \right)}, \quad i=0,1,\ldots,M
\end{align*}
and $M-1$ arbitrary but fixed angles, $\beta_i(\pi) =
(\beta_{i,1},\ldots,\beta_{i,M-1})$, out of $M$ angles defined by
\begin{align*}
  \arcsin \frac{\pi_j}{r_i(\pi)} = \arcsin
  \frac{\pi_j}{\sqrt{c_M \left(- \pi_i + \sum_{0\le k \le \ell
          \le M} \pi_k\pi_{\ell} \right) }}, \quad 0\le j\le M,\;
  j\not= i
\end{align*}
form a local polar coordinate system $(\beta_i(\pi),r_i(\pi))$ with
respect to the corner of the simplex $S^M\subset \R^M$ confined in
stopping region $\Gamma^{(i)}$, $1\le i \le M$ after a suitable linear
transformation $L$ into $\R^{M-1}$.

\section*{Acknowledgment}
The authors thank two anonymous referees and the associate editor for
their careful reading and valuable suggestions that improved the
readability of the manuscript.  The research of Savas Dayanik was
supported partially by the Air Force Office of Scientific Research,
under grant AFOSR-FA9550-06-1-0496 and by the U.S. Department of
Homeland Security through the Center for Dynamic Data Analysis for
Homeland Security administered through ONR grant number
N00014-07-1-0150 to Rutgers University.  The research of H.\ Vincent
Poor was supported in part by the U.S.\ Army Pantheon Project.

\appendix
\section{Proofs}

\subsection{Proof of Proposition \ref{P:BayesRiskInTermsOfPi}}

Note that since $\{\tau>n\} \in \F_n$ for every $n\ge 0$, we have
\begin{align*}
  \E\left[(\tau-\theta)^+\right]
  = \E\left[\sum_{n=0}^{\infty}\II{\theta \leq n<\tau}\right]
   = \sum_{n=0}^{\infty} \E[\II{\tau>n} \P\left(\theta\leq
     n\,|\,\F_n\right)]
    = \E\left[\sum_{n=0}^{\tau-1}
    \left(1-\Pi_n^{(0)}\right)\right].
\end{align*}
Moreover, for every $j\in \Mh$, we have $\{\tau=n, d= j\} \in \F_n$,
and $\E\left[\II{d=j}\II{\tau<\theta}\right]$ equals
\begin{multline*}
    \sum_{n=0}^{\infty} \E\left[\II{\tau=n,d=j}\II{\theta>n}\right]
    = \sum_{n=0}^{\infty} \E\left[\II{\tau=n,d=j} \Pi_n^{(0)} \right]
    = \lim_{N\to \infty} \sum_{n=0}^{N} \E\left[\II{\tau=n,d=j}
      \Pi_{\tau}^{(0)} \right] \\
    = \lim_{N\to \infty} \E\left[\sum_{n=0}^{N} \II{\tau=n,d=j}
      \Pi_{\tau}^{(0)} \right] = \lim_{N\to \infty} \E\left[
      \II{\tau\le N,d=j} \Pi_{\tau}^{(0)} \right] =
    \E\left[\II{\tau<\infty,d=j} \Pi_\tau^{(0)} \right]
    \end{multline*}
    because of the monotone convergence theorem and that $\lim_{N\to
      \infty} \{\tau \le N\} = \cup^{\infty}_{n=1} \{\tau \le n\} =
    \{\tau < \infty\}$; see, for example, Ross \cite{MR683455}.
    Similarly,
    $\E\left[\II{d=j,\mu=i}\II{\theta\leq\tau<\infty}\right]$ equals
\begin{align*}
  \sum_{n=0}^{\infty} \E\left[\II{\tau=n,d=j}\II{\theta\leq n,\mu=i}\right]
  = \sum_{n=0}^{\infty}
  \E\left[\II{\tau=n,d=j} \Pi_n^{(i)} \right] =
  \E\left[\II{\tau<\infty,d=j}\Pi_\tau^{(i)}\right],
    \end{align*}
    for every $i\in \Mh$.
    \begin{short-version}
  Plugging these expressions into \eqref{E:BayesRiskUnderP} completes
  the proof.
\end{short-version}
\begin{long-version}
  So, we can write the Bayes risk function in
  (\ref{E:BayesRiskUnderP}) as
  \begin{align*}
    R(\delta) &= \E\left[c(\tau-\theta)^+ + \sum_{j=1}^{M} \II{d=j}
      \left(a_{0j}\II{\tau<\theta} + \sum_{i=1}^{M} a_{ij}\II{\mu=i}
        \II{\theta\leq\tau<\infty}\right) \right]\\
    &= \E\left[\sum_{n=0}^{\tau-1}\left(1-\Pi_0^{(0)}\right) +
      \II{\tau<\infty}\sum_{j=1}^{M}\II{d=j}\left(a_{0j}
        \Pi_\tau^{(0)} +\sum_{i=1}^{M}a_{ij} \Pi_\tau^{(i)}
      \right)\right].
  \end{align*}
\end{long-version}
\qed

\subsection{Proof of Proposition \ref{P:PiProperties}}

Parts (a) and (b).
Fix any $ A = \{(X_1,\ldots,X_n)\in B \}\in \F_{n}$ for some
Borel $B\subset E^n.$ Then
(\ref{eq:finite-dimensional-distributions}) implies that
\begin{align}
  \P(A)
  = \int_{B} m(dx_1)\cdots m(dx_n) \alpha_n (x_1,\ldots,x_n)
  \label{E:integral-of-A}
\end{align}
where $\alpha_n (x_1,\ldots,x_n) \triangleq \sum_{i=0}^{M}
\alpha_n^{(i)}(x_1,\ldots,x_n)$, and
\begin{align*}
  \alpha_n^{(i)}(x_1,\ldots,x_n) \triangleq
\left\{
\begin{aligned}
&(1-p_0)(1-p)^n\prod_{l=1}^{n}f_0(x_l), && i=0,\\
& p_0\nu_i\prod_{j=1}^{n}
  f_i(x_j)+(1-p_0)p\nu_i \sum_{k=1}^{n}(1-p)^{k-1}\prod_{l=1}^{k-1}
  f_0(x_l)\prod_{j=k}^{n} f_i(x_j), && i\in \Mh.
\end{aligned}
\right.
\end{align*}
Hence, $\alpha_n(x_1,\ldots,x_n)$ is the joint probability density
function of $X_1,\ldots,X_n$ with respect to the measure
$m(dx_1)\cdots m(dx_n)$. Now for $i\in \Mh$,
\begin{multline*}
  \int_{A}\Pi_n^{(i)} d\P  = \E\left[ \I_A\II{\theta\leq n,\mu=i}
  \right]
  =  \int_{B} m(dx_1)\cdots m(dx_n)\; \alpha_n^{(i)}(x_1,\ldots,x_n)\\
  = \int_{B}\,m(dx_1)\cdots m(dx_n) \, \alpha_n (x_1,\ldots,x_n)
  \frac{\alpha_n^{(i)}(x_1,\ldots,x_n)}{ \alpha_n (x_1,\ldots,x_n)} =
  \int_{A}d\P\frac{\alpha_n^{(i)}(X_1,\ldots,X_n)}
  {\alpha_n(X_1,\ldots,X_n)}.
\end{multline*}
Hence,
\begin{align}
\label{eq:pi}
\Pi_n^{(i)} = \frac{\alpha_n^{(i)}(X_1,\ldots,X_n)}{\alpha_n
(X_1,\ldots,X_n)},\; i\in \Mh, \quad \text{and}\quad \Pi_n^{(0)} =
\frac{\alpha_n^{(0)}(X_1,\ldots,X_n)}{\alpha_n (X_1,\ldots,X_n)},
\end{align}
since $\sum_{i=0}^{M}\Pi_n^{(i)}=1$. Similar considerations also give
\begin{long-version}
\begin{align*}
  \int_{A}\Pi_n^{(i)} d\P &= \int_{A} \II{\theta\leq n,\mu=i} d\P
  =   \E\left[ \I_A\II{\theta\leq n,\mu=i} \right]\\
  &=  \sum_{k=0}^{n}\P\left\{\theta=k,\mu=i\right\}\P\left\{A \,|\, \theta=k,\mu=i\right\}\\
  &= p_0 \nu_i \int_{B} \prod_{j=1}^{n} f_i(x_j)
  \,m(dx_1)\cdots m(dx_n)\\
  &\quad + (1-p_0)p\nu_i\sum_{k=1}^{n} (1-p)^{k-1}
  \int_{B}\prod_{l=1}^{k-1} f_0(x_l)\prod_{j=k}^{n} f_i(x_j)
  \,m(dx_1)\cdots m(dx_n)\\
  &=  \int_{B}\alpha_n^{(i)}(x_1,\ldots,x_n)\,m(dx_1)\cdots m(dx_n).\\
  &= \int_{B}\,m(dx_1)\cdots
  m(dx_n)\left[\sum_{j=0}^{M}\alpha_n^{(j)}(x_1,\ldots,x_n)\right]
  \frac{\alpha_n^{(i)}(x_1,\ldots,x_n)}{\sum_{j=0}^{M}\alpha_n^{(j)}(x_1,\ldots,x_n)}\\
  &=
  \int_{A}d\P\frac{\alpha_n^{(i)}(X_1,\ldots,X_n)}{\sum_{j=0}^{M}\alpha_n^{(j)}(X_1,\ldots,X_n)},
\end{align*}
where the last two equalities follow by dividing and multiplying the
integrand $\alpha_n^{(i)}(x_1,\ldots,x_n)$ by the bracketed expression
in (\ref{E:integral-of-A}). Hence,
\begin{align*}
  \Pi_n^{(i)} &=
  \frac{\alpha_n^{(i)}(X_1,\ldots,X_n)}{\sum_{j=0}^{M}\alpha_n^{(j)}(X_1,\ldots,X_n)},\quad
  i=1,\ldots,M.
\end{align*}
Also, since $\sum_{i=0}^{M}\Pi_n^{(i)}=1$, we obtain
\begin{align*}
  \Pi_n^{(0)} &=
  \frac{\alpha_n^{(0)}(X_1,\ldots,X_n)}{\sum_{j=0}^{M}\alpha_n^{(j)}(X_1,\ldots,X_n)}.
\end{align*}
Next, let us calculate the joint conditional distribution
\begin{align*}
  \P\{\theta=k,\mu=i\,|\,\F_n\},\quad k\ge 0,\quad i\in \Mh
\end{align*}
of $\theta$ and $\mu$ given $\F_n=\sigma(X_1,\ldots,X_n)$ for $n\ge 0$
and $\F_0=\{\varnothing,\Omega\}$.  Again, fix any
$A=\{(X_1,\ldots,X_n)\in B\}\in\F_n$ for some Borel $B\in E^n$. Then
for $k\geq 0$ and $i=1,\ldots,M$,
\begin{align*}
  \int_{A} \P\{\theta=k,\mu=i\,|\,\F_n\} d\P
  &=  \P\left(\{\theta=k,\mu=i\} \cap A\right) \\
  &= \P\{\theta=k\}\nu_i\,\P\{A\,|\,\theta=k,\mu=i\}.
\end{align*}
Letting $k=0$ gives
\begin{align*}
  \int_{A} \P\{\theta=0,\mu=i\,|\,\F_n\} d\P
  &= p_0\nu_i\P\{A\,|\,\theta=0,\mu=i\}\\
  &= p_0\nu_i\int_{(x_1,\ldots,x_n)\in B}m(dx_1)\cdots m(dx_n)
  \prod_{j=1}^{n}f_i(x_j)\\
  &= \int_{B}m(dx_1)\cdots m(dx_n)
  \left[\sum_{j=0}^{M}\alpha_n^{(j)}(x_1,\ldots,x_n)\right]
  \frac{p_0\nu_i\prod_{j=1}^{n}f_i(x_j)}{\sum_{j=0}^{M}\alpha_n^{(j)}(x_1,\ldots,x_n)}\\
  &=
  \E\left[1_{A}\,\frac{p_0\nu_i\prod_{j=1}^{n}f_i(X_j)}{\sum_{j=0}^{M}\alpha_n^{(j)}(X_1,\ldots,X_n)}\right].
\end{align*}
Therefore,
\begin{align*}
  \P\{\theta=0,\mu=i\,|\,\F_n\}
  =\frac{p_0\nu_i\prod_{j=1}^{n}f_i(X_j)}{\sum_{j=0}^{M}\alpha_n^{(j)}(X_1,\ldots,X_n)}.
\end{align*}
By similar evaluations for different values of $k$, we obtain
\end{long-version}
\begin{align*}
  \P\{\theta=k,\mu=i\,|\,\F_n\} &= \left\{
    \begin{aligned}
      &\frac{p_0\nu_i}{\alpha_n(X_1,\ldots,X_n)}
      \prod_{j=1}^{n}f_i(X_j),      && k=0, \\
      &\frac{(1-p_0)(1-p)^{k-1}p\nu_i}
      {\alpha_n(X_1,\ldots,X_n)}\prod_{l=1}^{k-1}
    f_0(X_l)\prod_{j=k}^{n}f_i(X_j),    && 1\le k \le n, \\
    &\frac{(1-p_0)(1-p)^{k-1}p\nu_i}
    {\alpha_n (X_1,\ldots,X_n)}\prod_{l=1}^{n}f_0(X_l),    && k\geq n+1.\\
    \end{aligned}
  \right.
\end{align*}
Observe that for $k\geq n+1$,
\begin{align*}
  \P\{\theta=k,\mu=i\,|\,\F_n\}
  &=\frac{\alpha_n^{(0)}(X_1,\ldots,X_n)}
  {\alpha_n(X_1,\ldots,X_n)}\,(1-p)^{k-n-1}p\nu_i
=\Pi_{n}^{(0)}\,(1-p)^{k-n-1}p\nu_i.
\end{align*}
In particular, $\P\{\theta=n+1,\mu=i\,|\,\F_n\}
=\Pi_{n}^{(0)}\,p\nu_i$, and $\P\{\theta\leq n+1,\mu=i\,|\,\F_n\}$
equals
\begin{align*}
  \P\{\theta\leq n,\mu=i\,|\,\F_n\} + \P\{\theta=n+1,\mu=i\,|\,\F_n\}
  =\Pi_{n}^{(i)}+\Pi_{n}^{(0)}\,p\nu_i.
\end{align*}
Note also that $\P\{\theta>n+1\,|\,\F_n\}$ equals
\begin{align*}
  \sum_{k=n+2}^{\infty}\sum_{i=1}^{M}\Pi_{n}^{(0)}\,(1-p)^{k-n-1}p\nu_i
  =\Pi_{n}^{(0)}p\sum_{k=n+2}^{\infty}\,(1-p)^{k-n-1}
  =\Pi_{n}^{(0)}(1-p).
\end{align*}
Thus, $\E[\Pi_{n+1}^{(0)}\,|\,\F_n] =\P\{\theta>n+1\,|\,\F_n\}
  =\Pi_{n}^{(0)}(1-p) <\Pi_{n}^{(0)},$ and
\begin{align*}
  \E[\Pi_{n+1}^{(i)}\,|\,\F_n] &=\P\{\theta\leq n+1,\mu=i\,|\,\F_n\}
  =\Pi_{n}^{(i)}+\Pi_{n}^{(0)}\,p\nu_i >\Pi_{n}^{(i)},\quad i\in \Mh.
\end{align*}
Hence, $\{\Pi_n^{(0)},\F_n\}_{n\geq 0}$ is supermartingale, and
$\{\Pi_n^{(i)},\F_n\}_{n\geq 0}$, $i \in \Mh$ are submartingales.

For the proof of Part (c), note first that
\begin{align}
  \label{eq:unnormalized-probabilities}
  \alpha^{(i)}_{n+1} (x_1,\ldots,x_{n+1}) = \left\{
    \begin{aligned}
      & \left[\alpha^{(i)}_n(x_1,\ldots,x_n) + p \nu_i \alpha^{(0)}_n
        (x_1,\ldots,x_n) \right] f_i(x_{n+1}), && i \in \M,\\
      & (1-p) \alpha^{(0)}_n(x_1,\ldots,x_n) f_0(x_{n+1}), && i=0.
    \end{aligned}
  \right.
\end{align}
Substituting these expressions after writing $\Pi^{(i)}_{n+1}$, $i\in
\{0\} \cup \M$ by using (\ref{eq:pi}), and then dividing both
numerator and denominator by $ \alpha_n(X_1,\ldots,X_n)$ give
(\ref{E:Pi-Dynamics}).

Next, we find the conditional distribution of $X_{n+1}$ given $\F_n$
for $n\ge 0$. If $g:E\mapsto \R_+$ is a nonnegative function and
$A=\{(X_1,\ldots,X_n) \in B\} \in \F_n$, then $\int_A \E
[g(X_{n+1})\mid \F_n] d\P$ equals
\begin{multline*}
  \int_A g(X_{n+1}) d\P =\int_{B \times E} g(x_{n+1}) \alpha_{n+1}
  (x_1,\ldots,x_{n+1})\,  m(dx_1)\cdots m(dx_{n+1}) \\
 \begin{aligned}
   &=\int_{B} \left[\int_E g(x_{n+1}) \frac{\alpha_{n+1}
       (x_1,\ldots,x_{n+1})}{\alpha_{n} (x_1,\ldots,x_{n})}
     m(dx_{n+1})\right] \alpha_{n} (x_1,\ldots,x_{n})\, m(dx_1)\cdots
   m(dx_{n}) \\
   &=\int_{A} \left[\int_E g(x_{n+1}) \frac{\alpha_{n+1}
       (X_1,\ldots,X_n, x_{n+1})}{\alpha_{n} (X_1,\ldots,X_{n})}
     m(dx_{n+1})\right] d\P.
 \end{aligned}
\end{multline*}
Therefore, we have
\begin{align}
\label{eq:cond-distribution-of-X-n+1}
\begin{aligned}
  \E [g(X_{n+1}) \mid \F_n] = \int_E g(x) \frac{\alpha_{n+1}
    (X_1,\ldots,X_n, x)}{\alpha_{n} (X_1,\ldots,X_{n})} \, m(dx) =
  \int_E g(x) D(\Pi_n, x) m(dx),
\end{aligned}
\end{align}
where the second equality follows from (\ref{eq:pi}) after
substituting (\ref{eq:unnormalized-probabilities}) into previous
equality, and the mapping $D$ was defined by (\ref{eq:D-mappings}).
Then for every nonnegative function $f: S^M \mapsto \R_+$,
(\ref{E:Pi-Dynamics}) and (\ref{eq:cond-distribution-of-X-n+1}) imply
that
\begin{align*}
  \E[f(\Pi_{n+1})\mid \F_n] = \E\left[\left. f\left(\frac{D_0(\Pi_n,
          X_{n+1})}{D(\Pi_n, X_{n+1})},\ldots,\frac{D_M(\Pi_n,
          X_{n+1})}{D(\Pi_n, X_{n+1})}\right)\right| \F_n \right] =
  (\T f)(\Pi_n)
\end{align*}
in terms of the operator $\T$ defined by (\ref{E:T-operator}), and
$\E[f(\bPi_{n+1})|\F_n] = \E[f(\bPi_{n+1})\mid \bPi_n]$.
Therefore, the process $\{\Pi_n, \F_n;\; n\ge 0\}$ is Markov, and
the proof of part (c) is completed. \qed

\subsection{Proofs of Lemmas \ref{L:Vn-equal-expected-gamma} and
  \ref{L:backward-induction-eqns}}

Before proving the lemmas, we state Definition
\ref{D:directed-upwards}, Theorem \ref{T:directed-upwards}, and Lemma
\ref{L:directed-upwards} from Chow et al.\ \cite[pp.\
62-69]{MR0331675} for ease of reference.

\begin{definition}\label{D:directed-upwards}
  A collection $(\xi_t)_{t\in T}$ of random variables is called
  \emph{directed-upwards} if, for every $u, v \in T$, there exists
  $t\in T$ such that $\xi_t \geq \xi_u \vee \xi_v$.
\end{definition}

\begin{theorem}\label{T:directed-upwards}
  If a collection $(\xi_t)_{t\in T}$ of random variables is
  directed-upwards, then for every $t_0 \in T$, there exists a
  non-decreasing sequence $(\xi_{t_n})_{n\ge 0}$ in the
  collection $(\xi_t)_{t\in T}$ such that
  \begin{align*}
    \mathop{\emph{ess sup}}_{t\in T}\,\xi_t =
    \lim_{n\rightarrow\infty} \uparrow \xi_{t_n} \geq \xi_{t_0}\quad
    \text{almost surely}.
  \end{align*}
\end{theorem}
\begin{long-version}
\begin{proof}[Proof of Theorem \ref{T:directed-upwards}]
  Fix $t_0\in T$. Let $S$ be a countable subset of $T$ for which
  $\esup_{t\in T}\xi_t=\sup_{t\in S}\xi_t$ and enumerate its elements
  as $S=\{s_1,s_2,\ldots\}$.  Since $(\xi_t)_{t\in T}$ is
  directed-upwards, for every $n\geq 1$ there exists an index $t_n\in
  T$ such that $\xi_{t_n} \geq \xi_{t_{n-1}} \vee \xi_{s_n}$.  Thus,
  $(\xi_{t_n})_{n=0,1,\ldots}$ is a non-decreasing sequence and each
  element $\xi_{t_n}\geq\xi_{t_0}$.  By construction, $\sup_{t\in
    S}\xi_t\leq\sup_{n=0,1,\ldots}\xi_{t_n}$.  But since $\esup_{t\in
    T}\xi_t\geq\sup_{n=0,1,\ldots}\xi_{t_n}$, we have the equalities:
  $\esup_{t\in T}\xi_t = \sup_{n=0,1,\ldots}\xi_{t_n} =
  \lim_{n\rightarrow\infty}\xi_{t_n}$.
\end{proof}
\end{long-version}

\begin{lemma}\label{L:directed-upwards}
  For every $n\ge 0$. the collection
  $\{\E[Y_{\tau}\,|\,\F_n]\mid\tau\in C_n\}$ is directed-upwards.
\end{lemma}
\begin{long-version}
\begin{proof}[Proof of Lemma \ref{L:directed-upwards}]
  Fix $n$. Take $\tau_1,\tau_2\in C_n$, and define
  \begin{align*}
    \tau \triangleq \left\{
      \begin{aligned}
        &\tau_1, &&\text{on } A \triangleq\{\E[Y_{\tau_1}\,|\,\F_n]\geq\E[Y_{\tau_2}\,|\,\F_n]\}\\
        &\tau_2, &&\text{on } A^c
      \end{aligned}
    \right\}.
  \end{align*}
  Then, since $A\in\F_n$, we have $\E[Y_{\tau}\,|\,\F_n] =
  \E[Y_{\tau_1}\I_{A}+Y_{\tau_2}\I_{A^c}\,|\,\F_n] =
  \I_{A}\,\E[Y_{\tau_1}\,|\,\F_n]+\I_{A^c}\,\E[Y_{\tau_2}\,|\,\F_n] =
  \max\{\E[Y_{\tau_1}\,|\,\F_n],\E[Y_{\tau_2}\,|\,\F_n]\}$.
\end{proof}
\end{long-version}

\begin{proof}[Proof of Lemma \ref{L:Vn-equal-expected-gamma}]
  To prove the lemma, we establish two inequalities.  Note that
  $\gamma_n \geq \E[Y_{\tau}\,|\,\F_n]$ for all $\tau\in C_n$ by
  definition.  So, taking expectations, we obtain $\E\gamma_n \geq
  \sup_{\tau\in C_n}\E Y_{\tau} = -V_n$.  For the reverse direction,
  by Theorem~\ref{T:directed-upwards} and
  Lemma~\ref{L:directed-upwards} there exists a sequence of stopping
  times $(\tau_k)_{k\geq 1}\subset C_{n}$ such that $Y_{n} \leq
  \E[Y_{\tau_k}\,|\,\F_{n}] \uparrow \gamma_{n}$ as
  $k\rightarrow\infty$.  So, by the monotone convergence theorem, we
  have $\E\gamma_n = \E\left[\lim_{k\rightarrow\infty}
    \E[Y_{\tau_k}\,|\,\F_{n}]\right] = \lim_{k\rightarrow\infty} \E
  Y_{\tau_k} \leq -V_n$.  Proof of the equations $-V_n^N =
  \E\gamma_n^N$, $0\leq n\leq N$ is similar.
\end{proof}

\begin{proof}[Proof of Lemma \ref{L:backward-induction-eqns}]
  We have $\gamma_n \leq \max\{Y_{n},\E[\gamma_{n+1}\,|\,\F_n]\}$,
  because for every fixed $\tau\in C_n$, the expectation
  $\E[Y_{\tau}\,|\,\F_n]$ equals
  \begin{multline*}
    \E[Y_{\tau}\II{\tau=n}+Y_{\tau\vee(n+1)}\II{\tau\geq
      n+1}\,|\,\F_n]
     = Y_{n}\II{\tau=n}+\II{\tau\geq
       n+1}\E[\,\E[Y_{\tau\vee(n+1)}\,|\,\F_{n+1}]\,|\,\F_n]\\
    \leq Y_{n}\II{\tau=n}+\II{\tau\geq n+1}\E[\gamma_{n+1}\,|\,\F_n]
    \leq \max\{Y_{n},\E[\gamma_{n+1}\,|\,\F_n]\}.
  \end{multline*}

  For the reverse direction, note that $\gamma_n \geq Y_n =
  \E[Y_{n}\,|\,\F_n]$ by definition.  Since $\gamma_{n+1} =
  \esup_{\tau\in C_{n+1}}\E[Y_{\tau}\,|\,\F_{n+1}]$, by
  Theorem~\ref{T:directed-upwards} and Lemma~\ref{L:directed-upwards}
  there exists a sequence of stopping times $(\tau_k)_{k\geq 1}\subset
  C_{n+1}$ such that $Y_{n+1} \leq \E[Y_{\tau_k}\,|\,\F_{n+1}]
  \uparrow \gamma_{n+1}$ as $k\rightarrow\infty.$
  Since $C_{n+1}\subset C_{n}$ for all $n\geq 0$, we have $\gamma_{n}
  \geq \E[Y_{\tau_k}\,|\,\F_{n}] =
  \E[\,\E[Y_{\tau_k}\,|\,\F_{n+1}]\,|\,\F_{n}]$
  for all $k\geq 1$. Taking the limit as $k\rightarrow\infty$ and
  applying the monotone convergence theorem, we have
  $\gamma_n\geq\E[\gamma_{n+1}\,|\,\F_{n}]$.  Therefore,
  $\gamma_n\geq\max\{Y_{n},\E[\gamma_{n+1}\,|\,F_n]\}$.  By a similar
  argument, we can establish the other equations of
  Lemma~\ref{L:backward-induction-eqns}.
\end{proof}

\subsection{Proof of Lemma \ref{L:gamma-equals-gamma-prime}}
Because $(C^N_n)_{N\ge n}$ is increasing for every $n\ge 0$, the
sequence $(\gamma^N_n)_{N\ge n}$ is increasing for every $n\ge 0$ and
has a limit.  Set $\gamma'_n = \lim_{N\to \infty} \gamma^N_n$, $n\ge
0$. Because $\gamma^N_{n+1}\ge Y_{n+1}$ and $Y_{n+1}$ is integrable,
taking limits in $\gamma_n^N = \max\{Y_n,
\E[\gamma_{n+1}^N\,|\,\F_n]\}$, see Lemma
\ref{L:backward-induction-eqns}, and monotone convergence give
$\gamma'_n = \max\{Y_n, \E [\gamma'_{n+1} \mid \F_n]\}$ for every
$n\ge 0$.  Particularly, $(\gamma'_n)_{n\ge 0}$ is an
$\Fb$-supermartingale.

Obviously, $\gamma'_n \le \gamma_n$ for every $n\ge 0$. To prove
the reverse inequality, it is enough to show that $\gamma'_n \ge
\E[Y_{\tau}\mid \F_n]$ for every $\tau\in C_n$. Take any $\tau \in
C_n$. Then for every $F \in \F_n$ and $m\ge n$
\begin{multline*}
  \int_F \gamma'_n d\P = \int_{F\cap \{\tau = n\}} \gamma'_{\tau} d\P
  + \int_{F\cap \{\tau >n\}} \gamma'_n d\P \ge \int_{F\cap \{\tau =
    n\}} \gamma'_{\tau} d\P + \int_{F\cap \{\tau >n\}} \gamma'_{n+1}
  d\P \\
  \begin{aligned}
    & = \int_{F\cap \{n\le \tau \le n+1\}} \gamma'_{\tau} d\P +
    \int_{F\cap \{\tau >n+1\}} \gamma'_{n+1} d\P \ge \cdots \ge
    \int_{F\cap \{n\le \tau \le m\}} \gamma'_{\tau} d\P + \int_{F\cap
      \{\tau >m\}} \gamma'_{m} d\P,
  \end{aligned}
\end{multline*}
where the inequalities follow from $\Fb$-supermartingale property of
the process $(\gamma'_n)_{n \ge 0}$.  Because $\gamma'_k \ge Y_k$ for
every $k\ge 0$, we have $\gamma'_{\tau} \ge Y_{\tau}$, and for every
$m\ge n$
\begin{align*}
  \int_F \gamma'_n d\P \ge \int_{F\cap \{n\le \tau \le m\}} Y_{\tau}
  d\P + \int_{F\cap \{\tau >m\}} \gamma'_{m} d\P \ge \int_{F\cap
    \{n\le \tau \le m\}} Y_{\tau} d\P - \int_{F\cap \{\tau >m\}}
  (\gamma'_{m})^- d\P.
\end{align*}
Since $Y_{\tau}=-Y^-_{\tau}$ is integrable and $\tau<\infty$ a.s., we
have $\lim_{m\to \infty} \int_{F\cap \{n\le \tau \le m\}} Y_{\tau} d\P
= \int_F Y_{\tau} d\P$ by dominated convergence, and the proof will be
completed if $\varliminf_{m\rightarrow\infty}
\int_{\{\tau>m\}}(\gamma'_m)^- d\P = 0$. However, since
$\gamma'_m \ge Y_m$, we have $(\gamma'_m)^-\le Y^-_m$, and
$\int_{\{\tau>m\}}(\gamma'_m)^- d\P$ is less than or equal to
\begin{align*}
  \int_{\{\tau>m\}}Y_m^- d\P
  \leq \int_{\{\tau>m\}}m d\P + \|h\|\, \P\{\tau>n\} \leq
  \E\tau\II{\tau>n}+ \|h\|\, \P\{\tau>m\},
\end{align*}
where $\|h\|\triangleq \sup_{\bpi\in S^M}|h(\bpi)|$.  Since $h(\cdot)$
is bounded, $\E\tau <\infty$ and $\P\{\tau<\infty\}=1$, the right hand
side of the last inequality converges to zero as $n\rightarrow\infty$.
\qed

\subsection{Proof of Lemma \ref{L:pg-36}}

Fix any $N\ge 1$.  The equality holds trivially for $n=N$.  On the one
hand, the definition of the random variable
$\gamma^N_N$ in (\ref{eq:snell-envelopes}) implies that
\begin{align*}
  \gamma^N_N = \esup_{\tau\in C^N_N} \E[Y_{\tau} \mid \F_{N}] =
  \E[Y_{N} \mid \F_{N}] = Y_{N}
\end{align*}
because $C^N_N \equiv \{N\}$.  On the other hand, by the definition of
the operator $\M$ in (\ref{eq:operator-M}) we have $\M^0 h \equiv h$,
and
\begin{align*}
  -c\sum_{k=0}^{N-1}(1-\Pi_k^{(0)})-(\M^{N-N}h)(\Pi_n) =
  -c\sum_{k=0}^{N-1}(1-\Pi_k^{(0)})-h (\Pi_n) \equiv Y_N;
\end{align*}
thanks to (\ref{E:Yn}). Therefore,
(\ref{eq:markov-property-of-snell-envelope}) holds for $n=N$.
Now suppose that (\ref{eq:markov-property-of-snell-envelope}) is true
for some $n\geq 1$.  Then $\gamma_{n-1}^N =
\max\{Y_{n-1},\E[\gamma_n^N\,|\,\F_{n-1}]\}$ equals
\begin{multline*}
  \max\left\{-c\sum_{k=0}^{n-2}(1-\Pi_k^{(0)})-h(\bPi_{n-1}),\right.
  \left.\E\left[-c\sum_{k=0}^{n-1}(1-\Pi_k^{(0)})-(\M^{N-n}h)(\bPi_{n})\biggm|\F_{n-1}\right]\right\}\\
  \begin{aligned}
    &= -c\sum_{k=0}^{n-2}(1-\Pi_k^{(0)})
    -\min\left\{h(\bPi_{n-1}),c(1-\Pi_{n-1}^{(0)})+(\T(\M^{N-n}h))(\bPi_{n-1})\right\}\\
    &= -c\sum_{k=0}^{n-2}(1-\Pi_k^{(0)}) -(\M(\M^{N-n}h))(\bPi_{n-1})
    = -c\sum_{k=0}^{n-2}(1-\Pi_k^{(0)}) -(\M^{N-(n-1)}h)(\bPi_{n-1}).
  \end{aligned}
\end{multline*}
By induction, the equality holds for all $0\leq n \leq N$.
\qed

\subsection{Proof of Lemma \ref{L:pg-38}}
Applying Lemma~\ref{L:pg-36} for $n=0$ yields part (a) since
\begin{align}
  V_0^N = -\E \gamma_0^N = -\gamma_0^N = (\M^N h)(\bPi_0),\quad N\geq
  0.\label{E:V0N-MNh}
\end{align}
By Lemma~\ref{L:gamma-equals-gamma-prime}, $\gamma_n =
\lim_{N\rightarrow\infty}\gamma_n^N$, and so $V_n =
\lim_{N\rightarrow\infty}V_n^N$ by
Lemma~\ref{L:Vn-equal-expected-gamma} and the dominated convergence.
Since the left-hand side of \eqref{E:V0N-MNh} converges to $V_0$ as
$N\rightarrow\infty$, the limit of the right-hand side as
$N\rightarrow\infty$ exists and $V_0=\lim_{N\rightarrow\infty}(\M^N
h)(\bPi_0)$, which proves part (b).  \qed

\subsection{Proof of Lemma \ref{L:V-concave}}

Given $\bpi,\bpi'\in S^M$, $\lambda\in [0,1]$, and $\lambda'
\triangleq 1-\lambda$, we have
    \begin{multline*}
      \lambda(\T g)(\bpi)+\lambda'(\T g)(\bpi') = \lambda\int_{E}
      m(dx)\,D(\bpi,x)
      g\left(\frac{D_0(\bpi,x)}{D(\bpi,x)},\ldots,\frac{D_M(\bpi,x)}{D(\bpi,x)}\right)\\
     \begin{aligned}
       &\quad\quad+\lambda'\int_{E} m(dx)\,D(\bpi',x)
       g\left(\frac{D_0(\bpi',x)}{D(\bpi',x)},\ldots,\frac{D_M(\bpi',x)}{D(\bpi',x)}\right)\\
       &= \int_{E} m(dx)\,[\lambda D(\bpi,x)+\lambda'D(\bpi',x)]
       \left\{ \frac{\lambda D(\bpi,x)}{\lambda
           D(\bpi,x)+\lambda'D(\bpi',x)}
         g\left(\frac{D_0(\bpi,x)}{D(\bpi,x)},\ldots,\frac{D_M(\bpi,x)}{D(\bpi,x)}\right)
       \right.\\
       &\quad\quad\left.+\frac{\lambda'D(\bpi',x)}{\lambda
           D(\bpi,x)+\lambda'D(\bpi',x)}
         g\left(\frac{D_0(\bpi',x)}{D(\bpi',x)},\ldots,\frac{D_M(\bpi',x)}{D(\bpi',x)}\right)
       \right\}
    \end{aligned}
   \end{multline*}
Now, by the concavity of $g(\cdot)$ and the fact that
    \begin{align*}
        \frac{\lambda D(\bpi,x)}{\lambda D(\bpi,x)+\lambda'D(\bpi',x)}
        +\frac{\lambda'D(\bpi',x)}{\lambda
        D(\bpi,x)+\lambda'D(\bpi',x)} &= 1
    \end{align*}
is a convex combination, we continue the chain of inequalities to
obtain
\begin{multline*}
  \lambda(\T g)(\bpi)+\lambda'(\T g)(\bpi')
  \leq\int_{E} m(dx)\,[\lambda D(\bpi,x)+\lambda'D(\bpi',x)]\\
  \begin{aligned}
    &\quad\times g\left(\frac{\lambda
        D_0(\bpi,x)+\lambda'D_0(\bpi',x)} {\lambda
        D(\bpi,x)+\lambda'D(\bpi',x)}, \ldots, \frac{\lambda
        D_M(\bpi,x)+\lambda'D_M(\bpi',x)}
      {\lambda D(\bpi,x)+\lambda'D(\bpi',x)}\right)\\
    &=\int_{E} m(dx)\,[ D(\lambda\bpi+\lambda'\bpi',x)]
    \;g\left(\frac{D_0(\lambda\bpi+\lambda'\bpi',x)}
      {D(\lambda\bpi+\lambda'\bpi',x)}, \ldots,
      \frac{D_M(\lambda\bpi+\lambda'\bpi',x)}
      {D(\lambda\bpi+\lambda'\bpi',x)}\right)\\
    &= (\T g)(\lambda\bpi +\lambda'\bpi').
  \end{aligned}
\end{multline*}
Note that the second to last equality follows from the fact that
each of $D_0, \ldots, D_M, D$ is linear in its first argument. So,
we have established that $\T g$ is concave. \qed

\subsection{Proof of Proposition \ref{P:V-concave}}

Since $h(\bpi) = \min_{j\in \Mh} \sum_{i=0}^{M} \pi_i a_{ij}$ is
concave, and since the pointwise minimum of two concave functions
is concave, by Lemma \ref{L:V-concave} the function $(\M f)(\bpi)
= \min\{h(\bpi),c(1-\pi_0)+(\T f)(\bpi)\}$ is concave for every
bounded concave $f:S^M \mapsto \R$.  Therefore, $\M h,
\M^2h,\ldots$ are concave, and $V_0^0, V_0^1, \ldots$ are concave
by Lemma \ref{L:pg-38}(a). This proves part (a). For part (b),
note that Lemma \ref{L:pg-38}(b) implies that $V_0(\bpi) =
\lim_{N\rightarrow\infty} (\M^Nh)(\bpi)$ for every $\bpi\in S^M$;
thus, $V_0(\cdot)$ is concave on $S^M$.  \qed

\subsection{Proof of Proposition \ref{P:V-convergence-rate}}

The inequality $-V_0(\bpi) \geq -V_0^N(\bpi)$ for every $\bpi\in S^M$
and $N\ge 1$ is obvious.  Let us prove the second.  Fix $N\ge 1$,
$\bpi\in S^M$, and any $\varepsilon>0$.  Since
\begin{align*}
  0 \geq -V_0(\bpi) = \sup_{\tau\in C_0}\E_{\bpi}Y_{\tau} \geq
  \E_{\bpi}Y_0 \geq -\|h\| > -\infty
\end{align*}
is finite, there exists some stopping time $\tau_\varepsilon \in C_0$
such that
\begin{align}
  -V_0(\bpi)-\varepsilon < \E_{\bpi} Y_{\tau_\varepsilon}
  =\E_{\bpi}\left[-c\sum_{k=0}^{\tau_\varepsilon -
      1}(1-\Pi_k^{(0)})-h(\bPi_{\tau_\varepsilon})\right].\label{E:star}
\end{align}
Observe that $\tau_\varepsilon \wedge N \in C_0^N$ and
\begin{short-version}
    \begin{align}
        -V_0^N(\bpi)
            \geq \E_{\bpi} Y_{\tau_\varepsilon \wedge N}
            &\geq\E_{\bpi}\left[-c\sum_{k=0}^{\tau_\varepsilon -1}
                (1-\Pi_k^{(0)})-h(\bPi_{\tau_\varepsilon})\right]
                -\|h\|\,\P_{\bpi}\{\tau_\varepsilon\geq N\} \notag\\
            &\geq -V_0(\bpi)-\varepsilon -
                \frac{\|h\|}{N}\E_{\bpi}\tau_\varepsilon.\label{E:pound}
    \end{align}
\end{short-version}
\begin{long-version}
    \begin{align}
      -V_0^N(\bpi)
      &\geq \E_{\bpi} Y_{\tau_\varepsilon \wedge N}\notag\\
      &=\E_{\bpi}\left[-c\sum_{k=0}^{(\tau_\varepsilon \wedge N) -1}
        (1-\Pi_k^{(0)})-h(\bPi_{\tau_\varepsilon\wedge N})\right]\notag\\
      &\geq\E_{\bpi}\left[-c\sum_{k=0}^{\tau_\varepsilon -1}
        (1-\Pi_k^{(0)})-h(\bPi_{\tau_\varepsilon\wedge N})\right]\notag\\
      &=\E_{\bpi}\left[-c\sum_{k=0}^{\tau_\varepsilon -1}
        (1-\Pi_k^{(0)})-h(\bPi_{\tau_\varepsilon})\II{\tau_\varepsilon<N}
        -h(\bPi_{N})\II{\tau_\varepsilon\geq N} \right]\notag\\
      &\geq\E_{\bpi}\left[-c\sum_{k=0}^{\tau_\varepsilon -1}
        (1-\Pi_k^{(0)})-h(\bPi_{\tau_\varepsilon})\right]
      -\|h\|\,\P_{\bpi}\{\tau_\varepsilon\geq N\} \notag\\
      &\geq -V_0(\bpi)-\varepsilon -
      \frac{\|h\|}{N}\E_{\bpi}\tau_\varepsilon. \label{E:pound}
    \end{align}
\end{long-version}
The last inequality follows by the Markov inequality applied to
$\P_{\bpi}\{\tau_\varepsilon\geq N\}$ and since $\tau_\varepsilon$ is
$\varepsilon$-optimal for $V_0$. Next, we will bound
$\E_{\bpi}\tau_\varepsilon$ from above by using \eqref{E:star}:
    \begin{multline*}
        -\varepsilon-V_0(\bpi)
            < \E_{\bpi}\left[-c\sum_{k=0}^{\tau_\varepsilon -1}(1-\Pi_k^{(0)})
                -h(\bPi_{\tau_\varepsilon})\right]
            \leq\E_{\bpi}\left[-c\sum_{k=0}^{\tau_\varepsilon
            -1}(1-\Pi_k^{(0)})\right]\\
            = -c\E_{\bpi}\tau_\varepsilon +
                c\E_{\bpi}\sum_{k=0}^{\tau_\varepsilon -1}\Pi_k^{(0)}
            \leq -c\E_{\bpi}\tau_\varepsilon+
                c\E_{\bpi}\sum_{k=0}^{\infty}\Pi_k^{(0)}
            =-c\E_{\bpi}\tau_\varepsilon+
                c\sum_{k=0}^{\infty}\E_{\bpi}\Pi_k^{(0)}.
    \end{multline*}
\begin{short-version}
  Rearrangement after using the inequality $\E_{\bpi}\Pi_k^{(0)} \leq
  (1-p)^k$ of Proposition~\ref{P:PiProperties}(a) gives
    \begin{align*}
        \E_{\bpi}\tau_\varepsilon \leq
        \frac{1}{c}\left[V_0(\bpi)+\varepsilon\right] +
        \frac{1}{p} \leq \frac{\|h\|+\varepsilon}{c}+\frac{1}{p}.
    \end{align*}
\end{short-version}
\begin{long-version}
By Proposition~\ref{P:PiProperties}(a), we have
$\E_{\bpi}\Pi_k^{(0)} \leq (1-p)^k$ and so
    \begin{align*}
        -\varepsilon-V_0(\bpi)
            &\leq -c\E_{\bpi}\tau_\varepsilon+
                c\sum_{k=0}^{\infty}(1-p)^k
            =-c\E_{\bpi}\tau_\varepsilon+\frac{c}{p},
    \end{align*}
which can be rearranged as
    \begin{align*}
        \E_{\bpi}\tau_\varepsilon \leq
        \frac{1}{c}\left[V_0(\bpi)+\varepsilon\right] +
        \frac{1}{p} \leq \frac{\|h\|+\varepsilon}{c}+\frac{1}{p}.
    \end{align*}
\end{long-version}
Now using this bound on $\E_{\bpi}\tau_\varepsilon$ in
\eqref{E:pound} we have
    \begin{align*}
        -V_0^N(\bpi) \geq -V_0(\bpi)-\varepsilon -
        \frac{\|h\|}{N}\left(\frac{\|h\|+\varepsilon}{c}+\frac{1}{p}\right).
    \end{align*}
However, $\varepsilon$ was arbitrary, so taking the limit as
$\varepsilon\downarrow 0$ we obtain the desired bound.
\qed

\subsection{Proof of Proposition \ref{P:V0N-continuous}}

Recall that $ V_0^0(\bpi) = (\M^0h)(\bpi) = h(\bpi) = \min_{j\in
  \Mh}$ $\sum_{i=0}^{M}\pi_i a_{ij},$
which is continuous in $\bpi\in S^M$.  Suppose that
$V_0^N:S^M\mapsto\R_+$ is continuous for some $N\geq 0$.  Then by
\eqref{E:V0N-MNh}
    \begin{multline}
        V_0^{N+1}(\bpi)
            = (\M^{N+1}h)(\bpi) 
            = (\M V_0^N)(\bpi)
            = \min\left\{h(\bpi),c(1-\pi_0)+(\T V_0^N)(\bpi)\right\},\label{E:pg-55}
    \end{multline}
where (see \eqref{E:T-operator})
    \begin{align}
      (\T V_0^N)(\bpi)&=\int_{E}m(dx)\,D(\bpi,x)V_0^N\left(\frac{D_0(\bpi,x)}{D(\bpi,x)},\ldots,\frac{D_M(\bpi,x)}{D(\bpi,x)}\right).\label{E:TV0N}
    \end{align}
Note that
\begin{itemize}
\item the mapping $\bpi\mapsto D(\bpi,x)$ is continuous for every
  $x\in E$,
\item for every $x\in E$ such that $D(\bpi,x)>0$ (these are the
  $x$-values that matter in the defining integral of $(\T V_0^N)(\bpi)$
  above), the coordinates,
  $\frac{D_0(\bpi,x)}{D(\bpi,x)},\ldots,\frac{D_M(\bpi,x)}{D(\bpi,x)}$,
  are continuous,
\item since $V_0^N(\cdot)$ is continuous on $S^M$ by the induction
  hypothesis, the integrand in \eqref{E:TV0N} is continuous in $\bpi$
  for every fixed $x\in E$ such that $D(\bpi,x)>0$,
\item since $0\leq V_0^N(\cdot) \leq \|h\|$, the same nonnegative
  integrand is bounded from above by the integrable function
  $2\,\|h\|\sum_{i=0}^{M} f_i(x)$
  for every $\bpi\in S^M$,
\item then the mapping $\bpi\mapsto (\T V_0^N)(\bpi)$ is continuous by
  dominated convergence,
\item and finally, since $h(\bpi)$ and $c(1-\pi_0)+(\T V_0^N)(\bpi)$ are
  continuous, \eqref{E:pg-55} implies that the mapping $\bpi\mapsto
  V_0^{N+1}(\bpi)$ is continuous.
\end{itemize}
Hence, continuity holds for every $N\ge 0$ by induction, and this
completes the proof.  \qed

\subsection{Proof of Corollary \ref{C:V-continuous}}
The function $V_0(\bpi)$ on the compact space $S^M$ is the limit of
the sequence $\{V_0^N(\bpi)\}_{N\geq 0}$ of continuous functions,
\emph{uniformly} in $\bpi\in S^M$ by Proposition
\ref{P:V-convergence-rate}.  Therefore, it is continuous.  \qed

\subsection{Proof of Theorem \ref{T:Gamma-decreasing-subsets}}

By Lemmas \ref{L:pg-36} and \ref{L:Vn-equal-expected-gamma} we have
that $(V_0^N)_{N\geq 0}$ is a non-increasing sequence of functions,
bounded from above by the function $h$.  Since $h(\cdot)$ and
$h_j(\cdot), j\in \Mh$ are continuous and since $V_0^N(\cdot), N\geq0$
are continuous on $S^M$ by Proposition~\ref{P:V0N-continuous}, the set
$\Gamma_N^{(j)}=\{\bpi\in S^M\,|\,V_0(\bpi)=h(\bpi)=h_j(\bpi)\}$ is a
closed subset of $S^M$ for each $N\geq 0$ and $j\in \Mh$.

Fix $j\in \Mh$.  Then $V_0^{N+1}(\bpi)=h(\bpi)=h_j(\bpi)$ implies
$V_0^N(\bpi)=h(\bpi)=h_j(\bpi)$; and therefore, $\Gamma_{N+1}^{(j)}
\subset \Gamma_N^{(j)}$ for every $N\geq 0$.  Hence,
$(\Gamma_N^{(j)})_{N\geq 0}$ is a non-increasing sequence of closed
subsets of $S^M$.  Clearly, $\Gamma_N =
\bigcup_{j=1}^{M}\Gamma_N^{(j)}$, $N\geq 0$ and $(\Gamma_N)_{N\geq 0}$
is also a non-increasing sequence of closed subsets of $S^M$.
Moreover, since $V_0^N\searrow V_0$ by
Proposition~\ref{P:V-convergence-rate}, the limit of the
non-increasing sequence $(\Gamma_N)_{N\geq 0}$ is $\Gamma$; i.e.,
$\bigcap_{N=1}^{\infty} \Gamma_N = \Gamma$.  Similarly,
$\bigcap_{N=1}^{\infty} \Gamma_N^{(j)} = \Gamma^{(j)}$, $j\in \Mh$.

Given $\bpi\in S^M$, if the inequality $h_j(\bpi) \leq
\min\{h(\bpi),c(1-\pi_0)\}$ holds, then $h_j(\bpi) \leq h(\bpi)$,
which implies that $h_j(\bpi) = h(\bpi)$.  Also,
    \begin{align*}
        h_j(\bpi) &\leq \min\{h(\bpi),c(1-\pi_0)+(\T V_0)(\bpi)\} =
        V_0(\bpi).
    \end{align*}
This follows from the fact that $V_0\geq 0$ implies $\T V_0\geq 0$
and from the optimality equation of
Proposition~\ref{P:Dyn-prog-eqn}.  But, since $V_0\leq h$ on
$S^M$, we have $V_0(\bpi)=h_j(\bpi)= h(\bpi)$ and thus $\bpi\in
\Gamma^{(j)}$.  As a corollary, since $h_j(\e_j)=0 \leq
\min\{h(\e_j),c\}$, we have $\e_j\in\Gamma^{(j)}$.

In order to prove the convexity of $\Gamma_N^{(j)}$, take
$\bpi,\bpi'\in\Gamma_N^{(j)}$ and show that $\lambda\bpi +
(1-\lambda)\bpi'\in\Gamma_N^{(j)}$ for every $\lambda\in [0,1]$.
Since $V_0^N(\cdot)$ is concave by Proposition~\ref{P:V-concave},
we have
    \begin{multline*}
        \lambda V_0^N(\bpi) + (1-\lambda)V_0^N(\bpi')
            \leq V_0^N(\lambda\bpi+(1-\lambda)\bpi')
            \leq h(\lambda\bpi+(1-\lambda)\bpi')
            \leq h_j(\lambda\bpi+(1-\lambda)\bpi')\\
            = \lambda h_j(\bpi)+(1-\lambda)h_j(\bpi')
            = \lambda V_0^N(\bpi)+(1-\lambda)V_0^N(\bpi').
    \end{multline*}
Therefore, since $V_0^N(\bpi) \leq h(\bpi), \bpi\in S^M$, we have
    \begin{align*}
        V_0^N(\lambda\bpi+(1-\lambda)\bpi')
            &= h(\lambda\bpi+(1-\lambda)\bpi')
            = h_j(\lambda\bpi+(1-\lambda)\bpi')
    \end{align*}
and $\lambda\bpi+(1-\lambda)\bpi' \in \Gamma_N \cap \{\bpi\in
S^M\,|\,h(\bpi)=h_j(\bpi)\} = \Gamma_N^{(j)}$.  Hence,
$\Gamma_N^{(j)}$ is convex. Since an intersection of  convex sets
is again convex, $\Gamma^{(j)}=\bigcap_{N=1}^{\infty}
\Gamma_N^{(j)}$ is convex.

Thus, we have shown that $\Gamma = \bigcup_{i=1}^{M}\Gamma^{(i)}$
is the union of $M$ non-empty closed convex subsets of $S^M$.
Finally, consider
$\bpi(\lambda)\triangleq\lambda\e_0+(1-\lambda)\e_j$ for
$\lambda\in(0,\frac{c}{a_{0j}+c}]$.  Note that $c>0$ and
$a_{0j}\geq 0$ imply that the interval $(0,\frac{c}{a_{0j}+c}]$ is
non-empty.  The inequality $\lambda \leq \frac{c}{a_{0j}+c}$
implies that $c(1-\lambda) \geq \lambda a_{0j}=
h_j(\bpi(\lambda))$. Hence, $h(\bpi(\lambda))\leq
h_j(\bpi(\lambda)) \leq c(1-\lambda) \leq c(1-\lambda)+(\T
V_0)(\bpi(\lambda))$ and so $V_0(\bpi(\lambda))=h(\bpi(\lambda))$
by Proposition~\ref{P:Dyn-prog-eqn}.  Therefore,
$\Gamma\ni\bpi(\lambda)\notin\{\e_1,\ldots,\e_M\}$. \qed

\subsection{Proof of Lemma \ref{L:gamma-n-V}}

For every $n\geq 0$, the limit $\lim_{N\rightarrow\infty}
\gamma_n^N$ exists a.s.\ by Lemma \ref{L:gamma-equals-gamma-prime}.
So, fix $n$ and take the limit as $N\rightarrow\infty$ of the
expression in Lemma~\ref{L:pg-36}.  Then apply
Lemma~\ref{L:pg-38}(b) to obtain the result.
\qed

\subsection{Proof of Theorem \ref{T:sigma-properties}}

Let us prove part (a) first. Note that
\begin{short-version}
\begin{align*}
  \sigma &= \inf\{n\geq 0 \,|\, \bPi_n \in \Gamma\}
  =  \inf\{n\geq 0 \,|\, V_0(\bPi_n) = h(\bPi_n)\}
  = \inf\{n\geq 0 \,|\, \gamma_n = Y_n\}.
\end{align*}
\end{short-version}
\begin{long-version}
\begin{align*}
  \sigma &= \inf\{n\geq 0 \,|\, \bPi_n \in \Gamma\}
  =  \inf\{n\geq 0 \,|\, V_0(\bPi_n) = h(\bPi_n)\}\\
  &= \inf\left\{n\geq 0 \biggm|
    -c\sum_{k=0}^{n-1}(1-\Pi_k^{(0)})-V_0(\bPi_n)
    =-c\sum_{k=0}^{n-1}(1-\Pi_k^{(0)})-h(\bPi_n)\right\}\\
  &= \inf\{n\geq 0 \,|\, \gamma_n = Y_n\}.
\end{align*}
\end{long-version}
The second equality follows from the definition of $\Gamma$ and the
last equality follows from Lemma \ref{L:gamma-n-V} and the definition
of $Y_n$ \eqref{E:Yn}. Now, fix $n$ and recall from Lemma
\ref{L:backward-induction-eqns} that $\gamma_n =
\max\left\{Y_n,\E[\gamma_{n+1}|\F_n]\right\}$. Then
$\gamma_n=\E[\gamma_{n+1}|\F_n]$ on $\{\sigma>n\}$.  So,
\begin{align*}
  \E[\gamma_{(n+1)\wedge\sigma}\,|\,\F_n]
  &= \E[\gamma_{\sigma}\II{\sigma\leq n}\,|\,\F_n]
  +\E[\gamma_{n+1}\II{\sigma>n}\,|\,\F_n]\\
  &= \gamma_{\sigma}\II{\sigma\leq n}
  +\II{\sigma>n}\E[\gamma_{n+1}\,|\,\F_n] =
  \gamma_{\sigma}\II{\sigma\leq n} +\gamma_n\II{\sigma>n}
  =\gamma_{n\wedge\sigma}.
    \end{align*}
This establishes the martingale property of the stopped process
$\{\gamma_{n\wedge\sigma},\F_n\}_{n\geq 0}$.

To prove part (b), we use part (a) and Lemma
\ref{L:Vn-equal-expected-gamma} to write
\begin{short-version}
\begin{align*}
  -V_0 = \sup_{\tau\in C_0}\E Y_\tau = \gamma_0 =
  \E[\gamma_{n\wedge\sigma}]
  = \E[Y_{\sigma}\II{\sigma\leq n}] +\E[\gamma_{n}\II{\sigma>n}].
    \end{align*}
\end{short-version}
\begin{long-version}
\begin{multline*}
  -V_0 = \sup_{\tau\in C_0}\E Y_\tau = \gamma_0 =
  \E[\gamma_{n\wedge\sigma}]= \E[\gamma_{\sigma}\II{\sigma\leq n}]
   +\E[\gamma_{n}\II{\sigma>n}] \\
  = \E[Y_{\sigma}\II{\sigma\leq n}] +\E[\gamma_{n}\II{\sigma>n}].
    \end{multline*}
\end{long-version}
Since $Y_n = -\sum_{k=0}^{n-1}c(1-\Pi_k^{(0)})-h(\bPi_n)\leq 0$
for every $n$, we can use Fatou's Lemma after taking
$\limsup_{n\rightarrow\infty}$ of both sides to obtain
\begin{short-version}
\begin{align}
  -V_0
  \leq \E[Y_{\sigma}\II{\sigma<\infty}]
  +\E\left[(\limsup_{n\rightarrow\infty}\gamma_{n})
    \II{\sigma=\infty}\right].\label{E:pg-64}
\end{align}
Since $\limsup_{n\rightarrow\infty} \gamma_n
\leq\limsup_{n\rightarrow\infty}-\sum_{k=0}^{n-1}c(1-\Pi_k^{(0)}) =
-\infty$ by Remark ~\ref{R:PiProperties}, and $-V_0>-h>-\infty$, the
inequality \eqref{E:pg-64} implies that $\P\{\sigma=\infty\}=0$.
Therefore, the same inequality becomes $ -V_0 \equiv \sup_{\tau}\E
Y_{\tau} \leq \E Y_{\sigma}$. To show that $\sigma$ is optimal, we
must prove that $\sigma \in C_0$. Since $\sigma<\infty$ a.s., it is
enough to show $\E Y^-_{\sigma}<\infty$, which is equivalent to
showing that $\E\sigma <\infty$ by the discussion before equation
(\ref{E:OptimizationProblemC}).

However, since $\E Y_{\sigma}\ge -V_0>-\infty$, we also have $\E
\sigma<\infty$. Indeed,
\begin{align*}
  -\infty &< \E Y_{\sigma} =
  \E\left[-\sum_{k=0}^{\sigma-1}c(1-\Pi_k^{(0)})-h(\bPi_{\sigma})\right]
  \leq -c\E\sigma +
  c\E\left[\sum_{k=0}^{\infty}\Pi_k^{(0)}\right]\\
  &= -c\E\sigma + c\sum_{k=0}^{\infty}\E \Pi_k^{(0)} \leq -c\E\sigma +
  c\sum_{k=0}^{\infty}(1-p)^k = -c\E\sigma + \frac{c}{p}
\end{align*}
 implies $\E\sigma<\infty$.  Here, the last inequality follows
from Proposition~\ref{P:PiProperties}(a).  This completes the
proofs of parts (b) and (c). \qed
\end{short-version}
\begin{long-version}
\begin{align}
  -V_0 &\leq \limsup_{n\rightarrow\infty}\E[Y_{\sigma}\II{\sigma\leq
    n}]
  +\limsup_{n\rightarrow\infty}\E[\gamma_{n}\II{\sigma>n}]\notag\\
  &\leq \E[Y_{\sigma}\II{\sigma<\infty}]
  +\E\left[(\limsup_{n\rightarrow\infty}\gamma_{n})\II{\sigma=\infty}\right].\label{E:pg-64}
\end{align}
Note that
\begin{align*}
  \limsup_{n\rightarrow\infty} \gamma_n &=
  \limsup_{n\rightarrow\infty}
  \left(-\sum_{k=0}^{n-1}c(1-\Pi_k^{(0)})-V_0(\bPi_n)\right)\\
  &\leq\limsup_{n\rightarrow\infty} \left(-\sum_{k=0}^{n-1}
    c(1-\Pi_k^{(0)})\right) = -\infty.
\end{align*}
where the last equality above follows by Remark~\ref{R:PiProperties}.

Since $-V_0>-h>-\infty$ and $\limsup_{n\rightarrow\infty} \gamma_n =
-\infty$ a.s., the inequality \eqref{E:pg-64} implies that
$\P\{\sigma=\infty\}=0$. Therefore, the same inequality becomes
\begin{align*}
  -V_0 \equiv \sup_{\tau}\E Y_{\tau} \leq \E Y_{\sigma}.
\end{align*}
To show that $\sigma$ is optimal, we must prove that $\sigma \in C_0$.
Since $\sigma<\infty$ a.s., it is enough to show $\E
Y^-_{\sigma}<\infty$, which is equivalent to showing that $\E\sigma
<\infty$ by the discussion before equation
(\ref{E:OptimizationProblemC}).

However, since $\E Y_{\sigma}\ge -V_0>-\infty$, we also have $\E
\sigma<\infty$. Indeed,
\begin{align*}
  -\infty &< \E Y_{\sigma} =
  \E\left[-\sum_{k=0}^{\sigma-1}c(1-\Pi_k^{(0)})-h(\bPi_{\sigma})\right]\\
  &\leq -c\E\sigma + c\E\left[\sum_{k=0}^{\sigma-1}\Pi_k^{(0)}\right] \leq
  -c\E\sigma +
  c\E\left[\sum_{k=0}^{\infty}\Pi_k^{(0)}\right]\\
  &= -c\E\sigma + c\sum_{k=0}^{\infty}\E \Pi_k^{(0)} \leq -c\E\sigma
  + c\sum_{k=0}^{\infty}(1-p)^k = -c\E\sigma + \frac{c}{p}
\end{align*}
 implies $\E\sigma<\infty$.  Here, the last inequality follows
from Proposition~\ref{P:PiProperties}(a).  This completes the
proofs of parts (b) and (c). \qed
\end{long-version}

\bibliography{diagnosis}
\bibliographystyle{abbrv}

\end{document}